\documentclass{amsart} \usepackage[T1]{fontenc} 
\usepackage{graphicx}

\usepackage{amsmath,amssymb}  
	\usepackage[all]{xy}
	\usepackage{float}
	
	\newcommand{\eps}{\varepsilon}

	\newcommand{\Real}{\mathbb{R}}
	\newcommand{\Complex}{\mathbb{C}}
	\newcommand{\Integer}{\mathbb{Z}}
	\newcommand{\DBS}{\name{\mathcal{DBS}}}
	\newcommand{\Pass}{\name{\mathcal P}}
	\newcommand{\dM}{\partial M}
	\newcommand{\CJ}{\mathcal{CJ}^{\sim}}
	\newcommand{\CJcl}{\mathcal{CJ}_0}
	\newcommand{\CJfrak}{\mathfrak{CJ^{\sim}}}
	\newcommand{\DT}[1]{#1 \dots #1}
	\newcommand{\bydef}{\stackrel{\mbox{\tiny def}}{=}}
	\def \lmod#1\rmod {\left|\smash{#1}\right|\vphantom{#1}}
	\newcommand{\CSum}[1]{\operatorname{\mathcal C}\nolimits_{#1}^{\sim}}
	\newcommand{\pder}[2]{\frac{\partial #1}{\partial #2}}
	
	\newcommand{\pdertwo}[3]{\frac{\partial^2 #1}{\partial #2 \partial #3}}
	\newcommand{\name}[1]{\operatorname{\mathrm{#1}}}
	\newcommand{\one}{\name{\mathbf{1}}}
	\newcommand{\Upp}{\overline{\mathbb H}}
	
	\newcommand{\wM}{\widehat{M}}
	\newcommand{\wN}{\widehat{N}}
	\newcommand{\wf}{\widehat{f}}
	\newcommand{\T}{\mathcal T}
	\newcommand{\Tw}{\mathcal R}
	
	\theoremstyle{plain}
	\newtheorem{theorem}{Theorem}
	\newtheorem{lemma}[theorem]{Lemma}
	\newtheorem{proposition}[theorem]{Proposition}
	\newtheorem{corollary}[theorem]{Corollary}
	\theoremstyle{definition}
	\newtheorem{definition}[theorem]{Definition}
	\newtheorem{example}[theorem]{Example}
	
	\theoremstyle{remark}
	\def \remarkName {Remark}
	\newtheorem{remark}[theorem]{\remarkName}
	\newtheorem{Remark}{\remarkName}

	\numberwithin{theorem}{section}
	\numberwithin{equation}{section}

	\title{Ribbon decomposition and twisted Hurwitz numbers}
	\author{Yurii Burman}
	\address{Higher School of Economics, Moscow, Russia, and
		Independent University of Moscow.}
	\email{burman@mccme.ru}
	\thanks{The research was partially funded by the HSE University Basic Research
		Program and by the International Laboratory of Cluster Geometry NRU
		HSE (RF Government grant, ag.~No.~075-15-2021-608 dated
		08.06.2021). The research of the first-named author was also supported
		by the Simons Foundation IUM grant 2021.}
	\author{Rapha\"el Fesler}
	\address{Higher School of Economics, Moscow, Russia}
	\email{raphael.fesler@gmail.com}
	\dedicatory{In memoriam S.M.Natanzon.}
	
\begin{document}
		
		\begin{abstract}
			Ribbon decomposition is a way to obtain a surface with boundary
			(compact, not necessarily oriented) from a collection of disks by
			joining them with narrow ribbons attached to segments of the
			boundary. Counting ribbon decompositions gives rise to a ``twisted''
			version of the classical Hurwitz numbers (studied earlier in
			\cite{CD} in a different context) and of the cut-and-join
			equation. We also provide an algebraic description of these numbers
			and an explicit formula for them in terms of zonal polynomials.
		\end{abstract}
		
		\maketitle
		
		\section*{Introduction}
		
		A classical surgery in dimension $2$ studies connected sums of
		spheres, that is, ways to obtain a compact surface from a collection
		of spheres by gluing cylinders to them. In this paper we apply similar
		technique to surfaces with boundary: they are obtained from a
		collection of disks by gluing rectangles (``ribbons'') to their
		boundary. Like with the classical connected sum, to glue a ribbon one
		is to choose the orientation of the boundary at both points of gluing,
		so the ribbon glued may look twisted or not.
		
		Representation of a surface with boundary as a union of disks with the
		ribbons attached will be called its ribbon decomposition. See
		Fig.~\ref{Fg:CoverMob} for examples: the upper picture is a ribbon
		decomposition of an annulus, the lower one, of a Moebius band.
		
		Diagonals of ribbons form a graph embedded into the surface
		(a.k.a.\ fat graph, ribbon graph, combinatorial map, etc.), with all
		its vertices on the boundary. The edges adjacent to a given vertex are
		thus linearly ordered left to right (remember, an orientation of the
		boundary near every vertex is chosen); this ordering defines the
		embedding of the graph up to homotopy.
		
		Fix a positive integer $m$ and a partition $(\lambda_1 \DT\ge
		\lambda_s)$ of the number $n \bydef \lmod\lambda\rmod \bydef \lambda_1
		\DT+ \lambda_s$ into $s$ parts. The main object of study in this
		paper, the {\em twisted Hurwitz numbers} $h_{m,\lambda}^{\sim}$, have
		several definitions or models, as we call them. The first one, a
		topological model, uses ribbon decompositions. Denote by $\mathfrak
		S_{m,\lambda}$ the set of decompositions into $m$ ribbons of surfaces
		having boundary of $s$ components containing $\lambda_1 \DT,
		\lambda_s$ vertices (endpoints of ribbon diagonals). Then the twisted
		Hurwitz number is defined as
		\begin{equation}\label{Eq:DefHurw}
			h_{m,\lambda}^{\sim} \bydef \frac{1}{n!} \#\mathfrak S_{m,\lambda}
		\end{equation}
		
		Another model for $h_{m,\lambda}^{\sim}$ is algebraic. Consider a
		fixed-point-free involution
		\begin{equation}\label{Eq:DefTau}
			\tau = (1,n+1)(2,n+2) \dots (n,2n)
		\end{equation}
		in the symmetric group $S_{2n}$. Its centralizer is isomorphic to the
		reflection group of type $B_n$. Let $\sigma_1 \DT, \sigma_m \in
		S_{2n}$ be transpositions. A simple analysis (see Section
		\ref{Sec:Algebra} below) shows that the permutation
		\begin{equation}\label{Eq:HDecomp}
			u \bydef \sigma_1 \dots \sigma_m \tau \sigma_m \dots \sigma_1 \tau
			\in S_{2n}
		\end{equation}
		is decomposed into independent cycles as $u = c_1 c_1' \dots c_s c_s'$
		where $c_i' = \tau c_i \tau$ for every $i = 1 \DT, s$. Denote by
		$\mathfrak H_{m,\lambda}$ the set of sequences $(\sigma_1 \DT,
		\sigma_m)$ of $m$ transpositions such that the cycles $c_1 \DT, c_s$
		of the permutation $u$ of \eqref{Eq:HDecomp} have lengths $\lambda_1
		\DT, \lambda_s$. We prove (Theorem \ref{Th:RibbonToPermut}) that
		\begin{equation}\label{Eq:HurwAlgebr}
			h_{m,\lambda}^{\sim} = \frac{1}{n!} \#\mathfrak H_{m,\lambda}.
		\end{equation}
		
		The third model for $h_{m,\lambda}^{\sim}$ is algebro-geometric and is
		due to G.\,Chapuy and M.\,Dołęga \cite{CD}, who generalized the
		classical notion of a branched covering to the non-orientable
		case. Let $N$ denote a closed surface (compact $2$-manifold without
		boundary, not necessarily orientable), and $p: \wN \to N$, its
		orientation cover. Denote by $\Upp \bydef \Complex P^1/(z \sim
		\overline{z}) = \mathbb H \cup \{\infty\}$ where $\mathbb H \subset
		\Complex$ is the upper half-plane; its closure $\Upp$ is homeomorphic
		to a disk. Let $\pi: \Complex P^1 \to \Upp$ be the quotient map. A
		continuous map $f: N \to \Upp$ is called \cite{CD} a twisted branched
		covering if there exists a branched covering $\wf: \wN \to \Complex
		P^1$ (in the classical sense, a holomorphic map) such that $\pi \circ
		\wf = f \circ p$, and all the critical values of $\wf$ are real. These
		requirements imply in particular that the ramification profile of any
		critical value $c \in \Real P^1 \subset \Complex P^1$ of $\wf$ has
		every part repeated twice: $(\lambda_1, \lambda_1 \DT, \lambda_s,
		\lambda_s)$, and $\deg \widehat f = 2n$ is even. In this case we say
		that the ramification profile of the critical value $\pi(c) \in
		\partial\Upp$ of the map $f: N \to \Upp$ is $\lambda = (\lambda_1 \DT,
		\lambda_s)$.
		
		Twisted branched coverings are split into equivalence classes via
		right-left equivalence. Denote by $\mathfrak D_{m,\lambda}$ the set of
		equivalence classes of twisted branched coverings having $m$ critical
		values with the ramification profiles $2^1 1^{n-2}$ and one critical
		value $\infty$ with the ramification profile $\lambda$. Then
		\begin{equation}\label{Eq:HurwCD}
			h_{m,\lambda}^{\sim} = \frac{1}{n!} \#\mathfrak D_{m,\lambda}.
		\end{equation}
		
		Note that we prove equations \eqref{Eq:HurwAlgebr} and
		\eqref{Eq:HurwCD} differently. To prove \eqref{Eq:HurwAlgebr} we
		establish a direct one-to-one correspondence $\Xi$ between the sets
		$\mathfrak S_{m,\lambda}$ and $\mathfrak H_{m,\lambda}$. To prove
		\eqref{Eq:HurwCD} we show (Theorems \ref{Th:HurwCJ} and
		\ref{Th:CJtwisted}) that the generating function of the twisted
		Hurwitz numbers satisfies a PDE of parabolic type called twisted
		cut-and-join equation --- just like standard Hurwitz numbers, whose
		generating function satisfies the ``classical'' cut-and-join
		\cite{LandoKazarian}. Cardinalities of the sets $\mathfrak
		D_{m,\lambda}$ are shown in \cite{CD} to satisfy the same equation
		with the same initial data, so \eqref{Eq:HurwCD} follows. Finding a
		direct one-to-one correspondence between the sets $\mathfrak
		D_{m,\lambda}$ and $\mathfrak S_{m,\lambda}$ (or $\mathfrak
		H_{m,\lambda}$) is a challenging topic of future research.
		
		The paper contains three main sections in accordance with the three
		models described. In the first, ``topological'' section we study ribbon
		decompositions of surfaces with boundary (rigged with marked points)
		and the graphs (with numbered vertices and edges) formed by the
		diagonals of ribbons. The graphs appear to be $1$-skeleta of the
		surface, and the surface can be retracted to them (Theorem
		\ref{Th:Cell}); also, the graphs behave nicely under the orientation
		cover of the surface (Theorems \ref{Th:PropOrient} and
		\ref{Th:OrientViaGraph}).
		
		Graph embeddings into oriented surfaces were studied earlier in a
		number of works (see \cite{LandoZvonkin} for the general theory
		without boundary, \cite{GouldenYong} for the disk and
		\cite{BurmanZvonkine} for arbitrary surfaces and embeddings with a
		connected complement); they are in one-to-one correspondence with
		sequences of transpositions in the symmetric group. The cyclic
		structure of the product of the transpositions describes faces of the
		graph (i.e.\ connected components of its complement). The quantity of
		graphs with given faces is called a (classical) Hurwitz number and has
		been studied intensively during the last decades --- the research
		involving dozens of authors and hundreds of works; its thorough review
		is far outside the scope of this paper. The algebraic model for
		twisted Hurwitz numbers, studied in Section \ref{Sec:Algebra}, is a
		generalization of this correspondence. The section also contains an
		explicit formula for the cut-and-join equation (Theorem
		\ref{Th:CJtwisted}) and for the generating function of the twisted
		Hurwitz numbers (Theorem \ref{Th:HurwExpl}).
		
		In the last section we study the notion of the branched covering
		defined in \cite{CD} and show that they form an algebro-geometric
		model for twisted Hurwitz numbers.
		
		\subsection*{Acknowledgements}
		
		The research was partially funded by the HSE University Basic Research
		Program and by the International Laboratory of Cluster Geometry NRU
		HSE (RF Government grant, ag.~No.~075-15-2021-608 dated
		08.06.2021). The research of the first-named author was also supported
		by the Simons Foundation IUM grant 2021.
		
		The authors are grateful to G.\,Chapuy and M.\,Dołęga who explained
		them a broader perspective beyond the phenomena studied and also
		helped much with the combinatorics of the Laplace--Beltrami equation.
		
		We dedicate this article to the memory of our colleague Sergey
		Natanzon who fell victim of the COVID-19 pandemic. The subject of our
		research, to which Prof.~Natanzon was always attentive, matches some
		of his favourite scientific topics --- Hurwitz numbers and manifolds
		with boundary.
		
		\section{Surgery: a topological model for twisted Hurwitz numbers}
		\subsection{General definitions}
		
		\begin{definition}\label{Df:DBS}
			{\em Decorated-boundary surface} (DBS) is a triple $(M,(a_1 \DT,
			a_n), \linebreak (o_1 \DT, o_n))$ where $M$ is a compact surface
			($2$-manifold) with boundary, $a_1 \DT, a_n \in \dM$ are marked
			points and every $o_i$ is a local orientation of $\dM$ (hence, of
			$M$ itself, too) in the vicinity of the point $a_i$, such that
			\begin{itemize}
				\item every connected component of $M$ has nonempty boundary, and
				
				\item every connected component of $\dM$ contains at least one point
				$a_i$.
			\end{itemize}
		\end{definition}
		
		The DBS $M$ and $M'$ with the same number $n$ of marked points are
		called equivalent if there exists a homeomorphism $h: M \to M'$ such
		that $h(a_i) = a_i'$ and $h_*(o_i) = o_i'$ for all $i = 1 \DT, n$.
		The set of equivalence classes of DBS with $n$ marked points will be
		denoted $\DBS_n$.
		
		Pick marked points $a_i, a_j \in \dM$, and let $\eps_i, \eps_j \in
		\{+,-\}$. Consider points $a_i', a_j' \in \dM$ lying near $a_i, a_j$
		and such that the boundary segment $a_ia_i'$ is directed along the
		orientation $o_i$ if $\eps_i = +$ and opposite to it if $\eps_i = -$;
		the same for $j$. Now take a long narrow rectangle (``a ribbon''
		henceforth) and glue its short sides to $\dM$ as shown in
		Fig.~\ref{Fg:GlueRibbon}. The result of gluing is homeomorphic to a
		surface $M'$ with the boundary $\dM' \owns a_1 \DT, a_n$. The boundary
		of $M'$ near $a_i$ and $a_j$ contains a segment of $\dM$ (the ``old''
		part) and a segment of a long side of the ribbon glued (the ``new''
		part); define local orientations $o_i'$, $o_j'$ of $\dM'$ near $a_i$,
		$a_j$ so that the orientations of the ``old'' parts would be preserved
		(see bold curved arrows in Fig.~\ref{Fg:GlueRibbon}); for $k \ne i, j$
		take $o'_k = o_k$ by definition. Now $(M', (a_1 \DT, a_n), (o_1' \DT,
		o_n'))$ is a DBS, so we defined a mapping $G[i,j]^{\eps_i,\eps_j}:
		\DBS_n \to \DBS_n$ called {\em ribbon gluing}. The ribbon gluing
		$G[i,j]^{\eps_i,\eps_j}$ will be called twisted if $\eps_i \ne
		\eps_j$, and non-twisted otherwise; compare the left and the right
		picture in Fig.~\ref{Fg:GlueRibbon}.
		
		The inverse operation is called ribbon removal. To define it, take
		$\eps \in \{+,-\}$ and fix a smooth simple
		(i.e.\ non-selfintersecting) curve $\gamma$ on $M$ joining $a_i$ and
		$a_j$ and transversal to $\dM$ in its endpoints.  Take now a point
		$a_j' \in \dM$ near $a_i$ and $a_i' \in \dM$ near $a_j$ (NB the
		subscripts!) such that the segment $a_i a_j' \subset \dM$ is directed
		according to the orientation $o_i$ if $\eps = +$ and opposite to it if
		$\eps = -$, and consider a ``rectangle'' $\Pi$ like in
		Fig.~\ref{Fg:GlueRibbon}. Then $M' \bydef M \setminus \name{int}(\Pi)$
		is homeomorphic to a surface with the boundary $\dM' \owns a_1 \DT,
		a_n$. A local orientation $o_i'$ of $\dM'$ near $a_i$ is defined by
		the same rule as for the ribbon gluing: $o_i$ and $o_i'$ coincide on
		the intersection $\dM' \cap \dM$ near $a_i$; the same for $o_j'$, and
		also $o_k' \bydef o_k$ for all $k \ne i,j$. Now $(M', (a_1 \DT, a_n),
		(o_1' \DT, o_n'))$ is a DBS obtained from the original DBS by the
		operation $R[\gamma]^\eps$ of ribbon removal.
		
		\begin{remark}
			Local orientations $o_i$ and $o_j$ of $\dM$ define orientations of
			the normal bundle to $\gamma$; we call $\gamma$ {\em non-twisting}
			if the orientations are the same, and {\em twisting}
			otherwise. Obviously, the segment $a_j a_j'$ is directed along the
			orientation $o_j$ if $\eps = +$ and $\gamma$ is non-twisting or
			$\eps = -$ and $\gamma$ is twisting; otherwise $a_j a_j'$ is
			directed opposite to $o_j$.
		\end{remark}
		
		\begin{figure}
			\center
			\includegraphics[scale=0.5]{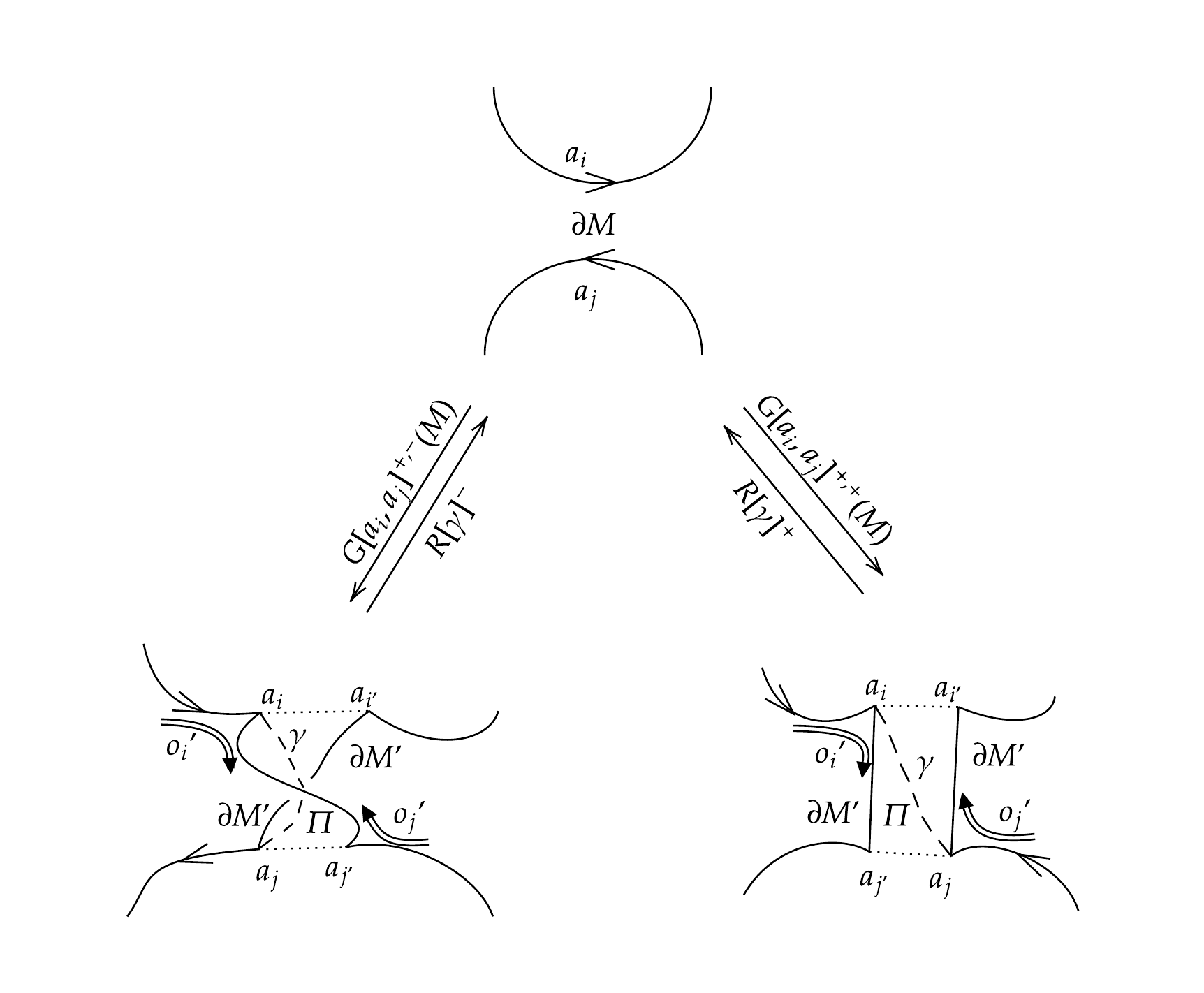} 
			\caption{Gluing ribbons}\label{Fg:GlueRibbon}
		\end{figure}
		
		The operation $R[\gamma]^\eps$ is a sort of inverse to ribbon
		gluing due to the following obvious statement:
		
		\begin{proposition}\label{Pp:Inv}
			\begin{enumerate}
				\item Let $i,j \in \{1 \DT, n\}$, $\eps_i, \eps_j \in \{+,-\}$ and
				$\gamma$ be a diagonal of the ribbon joining $a_i$ and $a_j$. Then
				$R[\gamma]^{\eps_i} G[i,j]^{\eps_i,\eps_j} = \name{id}_{\DBS_n}$.
				
				\item\label{It:Inv2} Let $\gamma$ be a simple smooth curve on $M$
				joining $a_i$ and $a_j$ and transversal to the boundary, and
				$\eps_i \in \{+,-\}$. Let $\eps_j \in \{+,-\}$ be defined as 
				$\eps_j = \eps_i$ if the curve $\gamma$ is
				non-twisting and $\eps_j = -\eps_i$ otherwise. Then
				$G[i,j]^{\eps_i,\eps_j} R[\gamma]^{\eps_i} = \name{id}_{\DBS_n}$.
			\end{enumerate}
		\end{proposition}
		
		\begin{remark}\label{Rm:Euler}
			Gluing a ribbon decreases the Euler characteristics of the surface
			by $1$ and removal, increases it by $1$.
		\end{remark}
		
		\subsection{Ribbon decompositions}
		
		By Definition \ref{Df:DBS} every connected component of a DBS contains
		a marked point. $M \in \DBS_n$ is called {\em stable} if every its
		connected component either contains at least two marked points or is a
		disk (with one marked point only).
		
		Denote by $E_n \in \DBS_n$ a union of $n$ disks with one marked point
		on the boundary of each.
		
		\begin{proposition}\label{EulerChar}
			$M \in \DBS_n$ is stable if and only if it can be obtained by gluing
			several ribbons to $E_n$. If $M$ is stable then its Euler
			characteristics $\chi(M)$ does not exceed $n$, and the number of
			ribbons is equal to $n - \chi(M)$.
		\end{proposition}
		
		\begin{proof}
			If a surface with a ribbon glued has a component with only one
			marked point then the gluing left this component intact. So, gluing
			a ribbon to a stable DBS keeps its stability, which proves the `only
			if' part of the proposition ($E_n$ is stable by definition).
			
			To prove the `if' part we will need a lemma:
			
			\begin{lemma}\label{Lm:NonSep}
				Let $n \ge 2$. Then for any $M \in \DBS_n$ which is connected and
				stable but is not a disk there exists a simple smooth
				nonseparating curve $\gamma$ joining two marked points.
			\end{lemma}
			
			``Nonseparating'' here means that the complement of $\gamma$ is
			connected, too.
			
			\begin{proof}[Proof of the lemma]
				$M$ contains at least two marked points. If $\dM$ is not connected
				then take two marked points on different components of $\dM$ and
				join them with a simple smooth curve $\gamma$; such curve is
				always nonseparating.
				
				Let now $\dM$ be connected. Then $M$ is a connected sum of a disk
				with several handles and/or Moebius bands. Let $S^1 \subset M$ be
				a circle separating the disk from a handle or from a Moebius band,
				and let $p, q \in S^1$ be two points. There exists a nonseparating
				curve $\delta$ inside the handle or the Moebius band joining $p$
				and $q$. Now pick a curve $\gamma_1$ joining $p$ with one marked
				point and $\gamma_2$ joining $q$ with another one. Then the union
				$\gamma \bydef \gamma_1 \cup \delta \cup \gamma_2$ is
				nonseparating as required.
			\end{proof}
			
			\begin{corollary}\label{Cr:DelStab}
				If $M \in \DBS_n$ is stable and $M \ne E_n$ then there exists a
				curve $\gamma$ on $M$ such that $M' \bydef R[\gamma]^\eps(M)$ is
				stable (regardless of $\eps$).
			\end{corollary}
			
			\begin{proof}[Proof of the corollary]
				A stable DBS different from $E_n$ contains a component with two or
				more marked points. If this component is a disk then take for
				$\gamma$ any simple curve joining these points. If it is not a
				disk then take for $\gamma$ the nonseparating curve of Lemma
				\ref{Lm:NonSep}.
			\end{proof}
			
			The proposition is now proved using induction on the Euler
			characteristic of $M$. Every component of $M$ is a manifold with
			nonempty bounbdary, so the $2$-nd Betti number of $M$ is zero and
			$\chi(M) = h_0(M) - h_1(M) \le h_0(M) \le n$; the equality is
			possible only if $M = E_n$. Let now $\chi(M) = n-m$, $m > 0$. Use
			Corollary \ref{Cr:DelStab} to obtain a curve $\gamma$ in $M$ such
			that $M' = R[\gamma]^+(M)$ is stable; by Remark \ref{Rm:Euler} one
			has $\chi(M') = n-m+1$, so by the induction hypothesis $M'$ can be
			obtained from $E_n$ by gluing $m-1$ ribbons. By assertion
			\ref{It:Inv2} of Proposition \ref{Pp:Inv} there exist $i$, $j$ and
			$\eps$ such that $M = G[i,j]^{+,\eps}(M')$ --- so, $M$ can be
			obtained by gluing $m$ ribbons.
		\end{proof}
		
		Let now, again, $M \in \DBS_n$ be obtained by gluing of $m$ ribbons to
		$E_n$:
		\begin{equation}\label{Eq:RibbonDecomp}
			M = G[i_m,j_m]^{\eps_m,\delta_m'} \dots
			G[i_1,j_1]^{\eps_1,\delta_1}E_n
		\end{equation}
		(that's what we will be calling a {\em ribbon decomposition} of
		$M$). For every ribbon, draw a diagonal joining its vertices $a_{i_k}$
		and $a_{j_k}$, and assign the number $k$ to it. The union of the
		diagonals is a graph $\Gamma \subset M$ with $m$ numbered edges $r_1
		\DT, r_m$ and the marked points $a_1 \DT, a_n$ as vertices; we call it
		a {\em diagonal graph} of the ribbon decomposition.
		
		Let $a_i$ be a marked point of $M$, $\Gamma \subset M$ be a diagonal
		graph of a ribbon decomposition, and let $\ell_1 \DT, \ell_k$ be the
		numbers of the edges of $\Gamma$ having $a_i$ as an endpoint, listed
		left to right according to the orientation $o_i$; denote $\Pass(a_i)
		\bydef (\ell_1 \DT, \ell_k)$.
		
		\begin{theorem}\label{Th:Prop}
			The diagonal graph $\Gamma$ has the following properties:
			\begin{enumerate}
				\item\label{It:Embed} (embedding) $\Gamma$ is embedded: its edges do
				not intersect one another or the boundary of $M$ except at
				endpoints.
				
				\item\label{It:Unimod} (anti-unimodality) For every vertex $a_i$ the
				sequence $\Pass(a_i) = (\ell_1 \DT, \ell_k)$ is anti-unimodal:
				there exists $p \le k$ such that $\ell_1 \DT> \ell_p \DT< \ell_k$.
				
				\item\label{It:TwoSide} (twisting rule) In the notation of the above
				call the edges $\ell_1 \DT, \ell_p$ negative at the endpoint
				$a_i$, and edges $\ell_p \DT, \ell_k$, positive (note that
				$\ell_p$ is both). Then any twisting edge of $\Gamma$ is positive
				at one of its endpoints and negative at the other one, and any
				non-twisting edge is either positive at both endpoints or negative
				at them.
				
				\item\label{It:Retract} (retraction) The graph $\Gamma$ is a
				homotopy retract of the surface $M$.
			\end{enumerate}
		\end{theorem}
		
		\begin{proof}
			Induction by the number $m$ of ribbons; the base $m = 0$ is
			obvious. For any $m > 0$, let $M = G[i_m,j_m]^{\eps_m \delta_m}M'$,
			and $\Gamma' \subset M' \subset M$ be the union of all the edges of
			$\Gamma$ except $m$.
			
			Assertion \ref{It:Embed}: the internal points of the edge $m$ lie in
			the interior of the ribbon $r_m = M \setminus M'$ and thus 
			belong neither to $\Gamma'$ nor to $\dM$.
			
			Assertion \ref{It:Unimod}: after gluing the ribbon $r_m$ to $M'$,
			the edge $m$ is either the leftmost or the rightmost of all the
			edges ending at $a_{i_m}$. Thus, if $\Pass(a_{i_m}) = (\ell_1 \DT,
			\ell_k)$ then either $\ell_1 = m$ and $\ell_2 \DT, \ell_k$ is
			anti-unimodal by the induction hypothesis, or $\ell_k = m$ and
			$\ell_1 \DT, \ell_{k-1}$ is anti-unimodal. In both cases $\ell_1
			\DT, \ell_k$ is anti-uninmodal.
			
			Assertion \ref{It:TwoSide} is true for the edges of $\Gamma' \subset
			M'$ by the induction hypothesis. Apparently, this is preserved after
			the ribbon $r_m$ is glued. The edge $m$ is the diagonal of $r_m$;
			the long sides of $r_m$ lie in $\partial M$, and therefore the
			edge $m$ is adjacent to $\partial M$ at both its endpoints, from the
			right for one of them and from the left for the other. This proves
			assertion \ref{It:TwoSide} for the edge $m$, too.
			
			To facilitate induction for assertion \ref{It:Retract}, we make it a
			bit stronger: fix, for every $i$, a small segment $e_i \subset \dM$,
			$a_i \in e_i$, and prove that there exists a homotopy retraction $f:
			M \to \Gamma$ such that $f(x) = a_i$ for all $x \in e_i$.
			
			By the induction hypothesis, such $f$ exists for $M'$ and
			$\Gamma'$. W.l.o.g.\ the ribbon $r_m$ containing $a_i$ and $a_j$ is
			glued to $M'$ so that its short sides lie within the segments $e_i$
			and $e_j$. Thus, the induction step is reduced to the following
			obvious statement: there exists a homotopy retraction of a rectangle
			$\Pi$ onto its diagonal $[ab]$ sending short sides and small
			neighbourhoods of the points $a, b \in \partial\Pi$ to the points
			$a$ and $b$.
		\end{proof}
		
		Let now $M \in \DBS_n$ and let $\Gamma \subset M$ be an embedded
		loopless graph with the vertices at the marked points of $M$ and the
		edges numbered $1 \DT, m$. We call $\Gamma$ {\em properly embedded} if
		it satisfies all the assertions of Theorem \ref{Th:Prop}: embedding,
		anti-unimodality, twisting rule and retraction. Connected components
		of the complement $M \setminus \partial M \setminus \Gamma$ will be
		called {\em faces}; connected components of $\partial M \setminus
		\{a_1 \DT, a_n\}$, {\em external edges}, and connected components of
		$\Gamma \setminus \{a_1 \DT, a_n\}$, {\em internal edges}.
		
		\begin{theorem}\label{Th:Cell}
			Vertices, edges and faces of a properly embedded graph $\Gamma$ form
			a cell decomposition of $M$ (as $0$-cells, $1$-cells and $2$-cells,
			respectively) such that every face is adjacent to exactly one
			external edge. The total number of faces is $n$.
		\end{theorem}
		
		\begin{proof}
			Let $\Gamma$ have $k$ faces $f_1 \DT, f_k$. Cover $M$ with the open
			subsets shown in Fig.~\ref{Fg:EulChar}.
			
			\begin{figure}
				\center
				\includegraphics[scale=1]{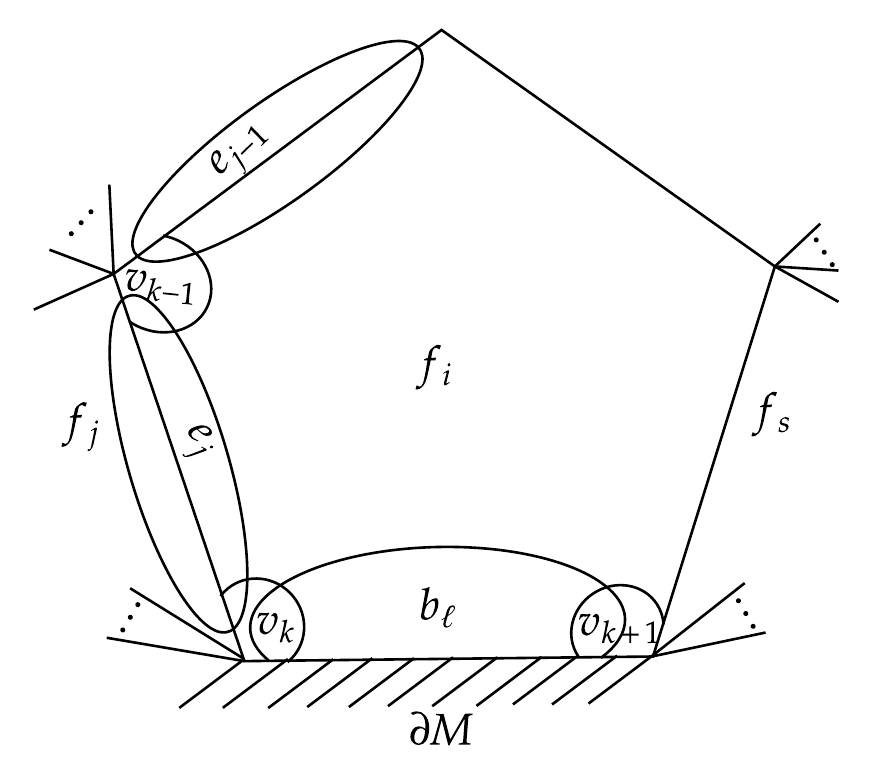} 
				\caption{The open cover of $M$}\label{Fg:EulChar}
			\end{figure}
			
			Sets $e_i$ (neighbourhoods of internal edges) are homeomorphic to
			disks, $b_i$ (neighbourhoods of external edges) and $v_i$
			(neighbourhoods of vertices), to half-disks; topology of faces $f_i$
			is yet to be described. Connected components of all the nonempty
			intersections of the sets (including faces) are homeomorphic to
			disks or half-disks, too.
			
			The nonempty intersections are:
			
			\begin{itemize}
				\item $2m$ connected components of $f_i \cap e_j$ for all $i, j$;
				\item $n$ components of $f_i \cap b_j$ for all $i, j$;
				\item If $\delta_j$ is the valency of 
				the $j$-th vertex of the graph, then there are
				$\delta_j+1$ components of $f_i \cap v_j$ for all $i$. The total
				number of components in $f_i \cap v_j$ is thus $\sum_j
				(\delta_j+1) = 2m + n$;
				\item $2m$ components of $e_i \cap v_j$, for all $i, j$;
				\item $2n$ components of $b_i \cap v_j$, for all $i, j$;
				\item $4m$ components of $f_i \cap e_j \cap v_k$, for all $i, j, k$;
				\item $2n$ components of $f_i \cap b_j \cap v_k$.
			\end{itemize}
			
			Thus the Euler characteristics of $M$ is
			\begin{align*}
				\chi(M) &= \sum_{i=1}^k \chi(f_i)+m+n+n-2m-n-(2m+n)-2m-2n+4m+2n\\
				&= \sum_{i=1}^k \chi(f_i)-m
			\end{align*}
			On the other hand, $\Gamma$ is a retract of $M$, so $\chi(M) =
			\chi(\Gamma) = n-m$, hence $\sum_{i=1}^k\chi(f_i) = n$. Faces are
			connected open $2$-manifolds, so $\chi(f_i) \le 1$ for every $i$,
			and therefore $n \le k$.
			
			Closure of a face is a compact manifold with boundary, so it cannot
			retract to its boundary. It means that the boundary of any face is
			not a subset of the graph and must contain an external edge. The
			total number of external edges is $n$, and an external edge belongs
			to the boundary of one face only. This implies $n \ge k$ and
			therefore $n = k$ and $\chi(f_i) = 1$ for every $i = 1 \DT, k$.
			
			So each $f_i$ is a disk. Its closure contains one external edge and
			$k_i$ internal ones, as well as vertices, so it is an image of the
			map $Q_i$ from some $(k_i+1)$-gon to $M$. Every $Q_i$ sends sides of
			the polygon to the edges and vertices to vertices, so collectively
			the $Q_i$, $i = 1 \DT, k$, are characteristic maps of a cell
			decomposition.
		\end{proof}
		
		Theorem \ref{Th:Cell} allows to prove the inverse of Theorem \ref{Th:Prop}:
		
		\begin{theorem}
			Let $M \in \DBS_n$ be stable and $\Gamma \subset M$ be a properly
			embedded graph. Then $\Gamma$ is the diagonal graph of a ribbon
			decomposition of $M$.
		\end{theorem}
		
		\begin{proof}
			Induction by the the number $m$ of edges of $\Gamma$. The base:
			$m=0$ means that $\Gamma$ consists of $n$ isolated vertices. Since
			$M$ is a retract of $\Gamma$, one has $M = E_n$.
			
			Let $m > 0$. The edge $e_m$ of $\Gamma$ joins the vertices $a_i$ and
			$a_j$ (necessarily different) and separates faces $f_p$ and $f_q$
			(which may be the same). By the anti-unimodality, $e_m$ is adjacent
			to $\dM$ at both $a_i$ and $a_j$. Using Theorem \ref{Th:Cell},
			consider a characteristic map $Q_p$ of the cell $f_p$. It maps the
			side $v_0 v_1$ of the polygon to the external edge of $f_p$ and the
			adjacent side $v_1v_2$, to $e_m$. Let $v' \in v_0v_1$ be a point
			near the vertex $v_1$, $a_i' \bydef Q_p(v') \in \partial M$;
			consider the image $T_p \bydef Q_p(v'v_1v_2) \subset M$ of the
			triangle $v'v_1v_2$. Then the union of $T_p$ and a similar triangle
			$T_q \subset f_q$ is a ribbon $H$ having $e_m$ as its diagonal.
			
			Let $\Gamma'$ be the graph $\Gamma$ with the edge $e_m$
			removed. Take $\eps = +$ if $\dM$ near $a_i$ is oriented towards
			$a_i'$, and $\eps = -$ otherwise. Then $\Gamma'$ is embedded into
			$M' \bydef R[e_m]^\eps(M)$; an immediate check shows that the
			embedding is proper, so $\Gamma'$ is the diagonal graph of a ribbon
			decomposition of $M'$ by the induction hypothesis. By Proposition
			\ref{Pp:Inv} $M$ can be obtained by gluing the ribbon $H$ to $M'$,
			which finishes the induction.
		\end{proof}
		
		\subsection{Oriented case and the orientation cover}
		
		A DBS $M$ is called {\em oriented} if all the local orientations $o_i$
		are consistent with a global orientation of the surface $M$. For an
		oriented $M$ the numbers of marked points read off the components of
		$\dM$ according to the orientation form a cyclic decomposition of some
		permutation $\sigma \in S_n$ called the {\em boundary permutation} of
		$M$ (here and below we denote by $S_n$ the permutation group). In
		other words, for any $k = 1 \DT, n$ the marked point adjacent to $a_k$
		in the positive direction of $\dM$ is $a_{\sigma(k)}$.
		
		It is easy to see that if $M$ is oriented and the gluing
		$G[i,j]^{\eps_i,\eps_j}$ is non-twisted ($\eps_i = \eps_j$) then
		$G[i,j]^{\eps_i,\eps_j}(M) \in \DBS_n$ is oriented, too. A ribbon
		decomposition
		\begin{equation}\label{Eq:OrientedRD}
			M = G[i_m,j_m]^{++} \dots G[i_1,j_1]^{++} E_n
		\end{equation}
		is called oriented; existence of such means, by obvious induction,
		that the surface $M$ is oriented.
		
		\begin{theorem}\label{Th:PropOrient}
			The diagonal graph $\Gamma$ of the oriented ribbon decomposition
			\eqref{Eq:OrientedRD} has the following properties (in addition to
			those granted by Theorem \ref{Th:Prop}):
			\begin{enumerate}
				\item (vertex monotonicity) For every vertex $a_i$ of $\Gamma$ the
				sequence $\Pass(a_i) = (\ell_1 \DT,\ell_k)$ is increasing: $\ell_1
				\DT< \ell_k$.
				\item (face monotonicity) For every face $f_i$ of $\Gamma$ the
				numbers $\ell_1 \DT, \ell_p$ of the internal edges adjacent to it
				are increasing if the count starts at the (only) external edge of
				$f_i$ and goes counterclockwise.
				\item (face separation) Every internal edge of $\Gamma$ separates
				two different faces.
				\item\label{It:BoundPerm} (boundary permutation) Let $a_{i_k}$ and
				$a_{j_k}$ be endpoints of the edge $e_k$ of $\Gamma$, $k = 1 \DT,
				m$. Then the boundary permutation of $M$ is equal to $(i_m j_m)
				\dots (i_1 j_1) \in S_n$.
			\end{enumerate}
		\end{theorem}
		
		\begin{proof}
			Vertex monotonicity is a particular case of anti-unimodality of
			Theorem \ref{Th:Prop}.
			
			If $\ell_j$ and $\ell_{j+1}$ are two internal edges on the boundary
			of $f_i$ sharing an endpoint $a$ then the orientation of the
			boundary near $a$ is consistent with the counterclockwise
			orientation of $f_i$. Then the vertex monotonicity implies $\ell_j <
			\ell_{j+1}$, which proves face monotonicity. The face monotonicity
			implies, in its turn, the face separation: as one moves around a
			face, the numbers of the internal edges seen are increasing and
			therefore cannot repeat.
			
			Let $a_k, a_s \in \dM$ be neighbouring vertices, that is, the
			endpoints of an external edge. By Theorem \ref{Th:Cell} and the face
			monotonicity, this is the sole external edge of a face $f$, its
			remaining sides being internal edges numbered $\ell_1 \DT< \ell_p$,
			as one moves from $a_k$ to $a_s$. Consider an action of $S_n$ on the
			vertices of $M \in \DBS_n$ by permuting their numbers; in
			particular, the transposition $(i_t j_t)$ exchanges the numbers of
			the vertices joined by the $t$-th edge of the diagonal graph,
			leaving the other vertices intact. So, the transposition
			$(i_{\ell_1}j_{\ell_1})$ moves $a_k$ to its neighbour at the face
			$f$; then the transposition $(i_{\ell_2},j_{\ell_2})$ (where $\ell_2
			> \ell_1$, so it is applied {\em after} the first one) moves it to
			the next vertex of the same face, etc.; eventually, $\sigma = (i_m
			j_m) \dots (i_1 j_1)$ moves $a_k$ to $a_s = a_{\sigma(k)}$.
		\end{proof}
		
		Every manifold $M$ (possibly with boundary) has the orientation cover,
		uniquely defined up to an obvious isomorphism: it is an oriented
		manifold $\wM$ of the same dimension together with a fixed-point-free
		smooth involution $\T: \wM \to \wM$ reversing the orientation and such
		that $M$ is diffeomorphic to its orbit space.
		
		The quotient map $\wM \to \wM/T = M$ is a $2$-sheeted covering,
		trivial iff $M$ is oriented. For $2$-manifolds with boundary there is
		
		\begin{lemma}\label{Lm:OrientCov}
			The orientation covering is trivial over the boundary of a
			$2$-manifold.
		\end{lemma}
		
		\begin{proof}
			The boundary $\partial M$ and its cover $\partial\wM$ are unions of
			circles. If the covering is nontrivial over the boundary then there
			is a component $C \subset \partial M$ covered by a $\T$-invariant
			circle $C' \subset \partial\wM$.
			
			A continuous map $A: S^1 \to S^1$ has at least $\lmod \deg A-1\rmod$
			fixed points, so the fixed-point-free map $\T: C' \to C'$ has degree
			$1$ and therefore, being a covering, preserves orientation. Since
			$C' \subset \partial \wM$, it means that $\T: \wM \to \wM$ also
			preserves local orientation at every point $a \in C'$. But $\T$ is
			orientation-reversing everywhere --- a contradiction.
		\end{proof}
		
		Let a fixed-point-free involution $\tau \in S_{2n}$ be defined by
		\eqref{Eq:DefTau}. The notion of an orientation cover can be extended
		to decorated-boundary surfaces as follows: $\wM \in \DBS_{2n}$ with
		the marked points $b_1 \DT, b_{2n}$ is called the orientation cover of
		$M \in \DBS_n$ with the marked points $a_1 \DT, a_n$ if $\wM$ is
		oriented and there exists a fixed-point-free smooth involution $\T:
		\wM \to \wM$ reversing the orientation and such that $\T(b_k) =
		b_{\tau(k)}$ for all $k = 1 \DT, 2n$, and also there exists a
		diffeomorphism $p: \wM/\T \to M$ between the orbit space and $M$ such
		that $p(\{b_k,b_{\tau(k)}\}) = a_k$ for all $k = 1 \DT, n$.
		
		For $M \in \DBS_n$ the surface $\wM$ and involution $\T: \wM \to \wM$
		are uniquely defined; the marked points are $p^{-1}(a_1) \bydef \{b_i,
		b_{i+n}\} \subset \partial\wM$. The numbering of the two points $b_i$
		and $b_{i+n}$ depends on the local orientation $o_i$ of $\dM$ at $a_i$
		and is fixed by the following rule: the mapping $p: \partial\wM \to
		\dM$ preserves the orientation at $b_i$ and reverses it at $b_{i+n}$,
		$i = 1 \DT, n$. Thus, for every $M \in \DBS_n$ an orientation cover
		$\wM \in \DBS_{2n}$ is unique.
		
		Let $1 \le i \le n$ and $\eps \in \{+,-\}$. Denote $i^\eps = 
		\begin{cases}
			i, &\eps = +,\\
			\tau(i), &\eps = -.
		\end{cases}$
		
		\begin{figure}
			\center
			\includegraphics[scale=0.8]{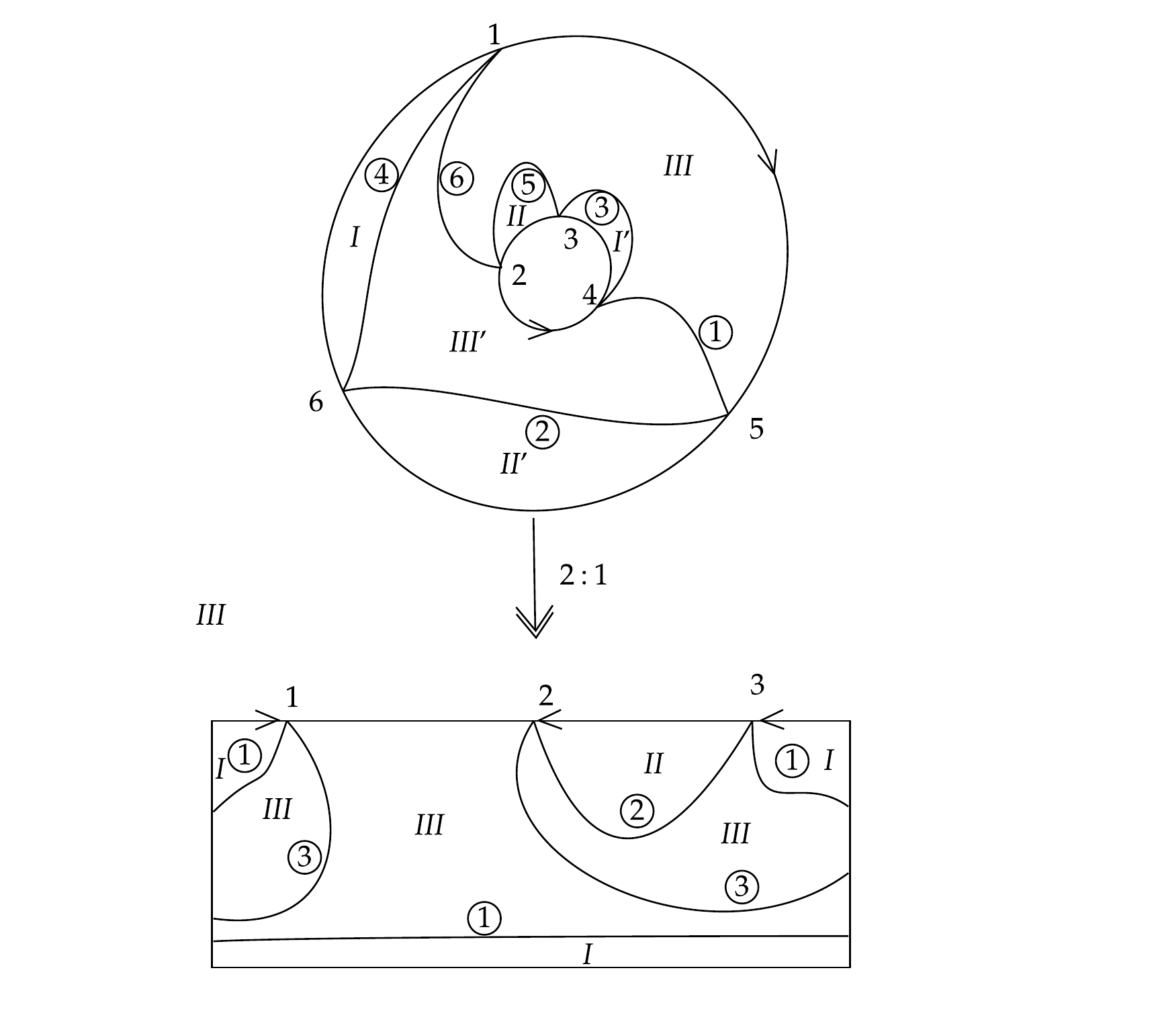} 
			\caption{Covering of the Moebius band\\
				$M= G[1,2]^{++} G[2,3]^{++} G[1,3]^{+-} E_3$ by an
				annulus}\label{Fg:CoverMob}
		\end{figure} 
		
		\begin{theorem}\label{Th:OrientViaGraph}
			Let $M$ be a DBS of equation \eqref{Eq:RibbonDecomp}. Then its
			orientation cover is
			\begin{equation}\label{Eq:2Cover}
				\wM = G[i_m^{\eps_m} j_m^{\delta_m}]^{++} \dots G[i_1^{\eps_1}
				j_1^{\delta_1}]^{++} G[i_1^{-\eps_1} j_1^{-\delta_1}]^{++} \dots
				G[i_m^{-\eps_m} j_m^{-\delta_m}]^{++} E_n.
			\end{equation}
			The involution $\T: \wM \to \wM$ maps the ribbon $r_\ell$ to the
			ribbon $r_{2m+1-\ell}$ for all $\ell = 1 \DT, 2m$.
		\end{theorem}
		
		See Figure \ref{Fg:CoverMob} for an example. The Roman numerals there
		mean faces, Arabic numbers, vertices, and the circled numbers mark
		diagonals of the ribbons.
		
		\begin{proof}
			Let $a_i$ be a marked point of $M$ with $\Pass(a_i) = (\ell_1 \DT,
			\ell_k)$ where $\ell_1 \DT> \ell_p \DT< \ell_k$, and let $b_i,
			b_{\tau(i)} \in \wM$ be its preimages. Use the induction on $m$ to
			prove the theorem showing simultaneously that $\Pass(b_i) =
			(m+1-\ell_1 \DT, m+1-\ell_p, m+\ell_{p+1} \DT, m+\ell_k)$ and
			$\Pass(b_{\tau(i)}) = (m+1-\ell_k \DT, m+1-\ell_{p+1}, \ell_p+m \DT,
			\ell_1+m)$.
			
			The base $m=0$ is obvious. For $m > 0$ let $M = G[i,j]^{\eps \delta}
			M'$ where $i,j,\eps,\delta$ mean $i_m, j_m, \eps_m, \delta_m$,
			respectively. If $\Pass_{M'}(a_i) = (\ell_1 \DT, \ell_k)$ where
			$\ell_1 \DT> \ell_p \DT< \ell_k$ then $\Pass_M(a_i) = (\ell_1 \DT,
			\ell_k, m)$ if $\eps = +$ and $\Pass_M(a_i) = (m, \ell_1 \DT,
			\ell_k)$ if $\eps = -$; the same for $a_j$ (depending on $\delta$
			instead of $\eps$).
			
			Denote by $\wM'$ the orientation cover of $M'$ and define $\wM$ by
			\eqref{Eq:2Cover}. By the induction hypothesis $\wM'$ is a subset of
			$\wM$ (a union of all the ribbons except $r_1$ and $r_{2m}$). Extend
			$\T: \wM' \to \wM'$ to the involution $\wM \to \wM$ sending $r_1$ to
			$r_{2m}$ and vice versa; also extend the homeomorphism $\rho:
			\wM'/\T \to M'$ to a map $\wM/\T \to M$ sending $r_1$ and $r_{2m}$
			to the $m$-th ribbon of $M$. To complete the proof we are to check
			that the extended $\T$ and $\rho$ are continuous on the boundary of
			the ribbons $r_1$ and $r_{2m}$.
			
			By the induction hypothesis, $\Pass_{\wM'}(b_i) = (m-\ell_1 \DT,
			m-\ell_p, \ell_{p+1}+m-1 \DT, \ell_k+m-1)$ and
			$\Pass_{\wM'}(b_{\tau(i)}) = (m-\ell_k \DT, m-\ell_{p+1}, \ell_p+m-1
			\DT, \ell_1+m-1)$. So, if $\eps = +$ then $\Pass_{\wM}(b_i) =
			(m+1-\ell_1 \DT, m+1-\ell_p, \ell_{p+1}+m \DT, \ell_k+m, 2m)$ and
			$\Pass_{\wM'}(b_{\tau(i)}) = (1, m+1-\ell_k \DT, m+1-\ell_{p+1},
			\ell_p+m \DT, \ell_1+m)$, and if $\eps = -$ then $\Pass_{\wM}(b_i) =
			(1, m+1-\ell_1 \DT, m+1-\ell_p, \ell_{p+1}+m \DT, \ell_k+m)$ and
			$\Pass_{\wM'}(b_{\tau(i)}) = (m+1-\ell_k \DT, m+1-\ell_{p+1},
			\ell_p+m \DT, \ell_1+m, 2m)$; the same for $b_j$ and $b_{\tau(j)}$,
			with $\delta$ instead of $\eps$.
			
			Thus, if $\eps = +$ then the ribbon $r_{2m}$ is adjacent to
			$r_{\ell_k+m}$ and the ribbon $r_1$, to $r_{m+1-\ell_k}$; the $m$-th
			ribbon of $M = G[i,j]^{\eps \delta} M'$ is adjacent to the ribbon
			numbered $\ell_k$. By the induction hypothesis, $\T$ exchanges
			$r_{\ell_k+m}$ and $r_{m+1-\ell_k}$, so the extensions of $\T$ and
			$\rho$ are continuous on the ``long'' sides of $r_{2m}$ and $r_1$
			containing $b_i$ and $b_{\tau(i)}$, respectively. The proof in the
			case $\eps = -$ is the same. A similar analysis of $\Pass(b_j)$ and
			$\Pass(b_{\tau(j)})$ for $\delta = +$ and $\delta = -$ shows that
			$\T$ and $\rho$ are continuous on the other sides of $r_{2m}$ and
			$r_1$, too.
		\end{proof}
		
		\section{Algebraic model and twisted cut-and-join equation}\label{Sec:Algebra}
		
		\subsection{Algebraic preliminaries and twisted Hurwitz numbers}\label{SSec:Prelim}
		
		Recall \cite{Coxeter} that the reflection group $B_n$ is the
		semidirect product $S_n \ltimes (\Integer/2\Integer)^n$ where $S_n$
		acts on $(\Integer/2\Integer)^n$ by permuting the factors; it is
		generated by reflections $s_{ij} \bydef (ij) \ltimes \one$ and $l_i
		\bydef \name{id} \ltimes \one_i$ where $\one \bydef (1 \DT, 1)$ and
		$\one_i \bydef (1 \DT, -1 \DT, 1)$ ($-1$ at the $i$-th place).  The
		group $B_n$ can be embedded into $S_{2n}$ as a centralizer $C(\tau)$
		where $\tau$, as above, is defined by \eqref{Eq:DefTau}; the
		isomorphism $A: C(\tau) \to B_n$ is given by $A(\sigma) =
		\lambda_\sigma \ltimes (\eps_\sigma^{(1)} \DT, \eps_\sigma^{(n)})$
		where
		\begin{equation*}
			\lambda_\sigma(i) = 
			\begin{cases}
				\sigma(i), &\sigma(i) \le n,\\    
				\sigma(i)-n, &\sigma(i) \ge n+1
			\end{cases} \qquad \text{and} \qquad \eps_\sigma^{(i)} = 
			\begin{cases}
				1, &\sigma(i) \le n, \\
				-1, &\sigma(i) \ge n+1.
			\end{cases}
		\end{equation*}
		
		Let $C^{\sim}(\tau) \bydef \{\sigma \in S_{2n} \mid \tau \sigma =
		\sigma^{-1} \tau\}$ (a ``twisted centralizer'' of $\tau$).
		
		\begin{lemma}
			Let $\sigma = c_1 \dots c_m \in C^{\sim}(\tau)$ where $c_1 \DT, c_m$ are
			independent cycles. Then for every $i$
			\begin{itemize}
				\item either there exists $j \ne i$ such that $c_i = (u_1 \dots
				u_k)$ and $c_j = (\tau(u_k) \dots \tau(u_1))$;
				\item or $c_i$ has even length $2k$ and looks like $c_i = (u_1 \dots
				u_k \tau(u_k) \dots \tau(u_1))$. 
			\end{itemize}
		\end{lemma}
		
		In the first case we say that the cycles $c_i$ and $c_j$ are
		$\tau$-symmetric, and in the second case the cycle $c_i$ is
		$\tau$-self-symmetric.
		
		\begin{proof}
			Let $c_i = (u_1^{(i)} \dots u_{k_i}^{(i)})$ for all $i = 1 \DT,
			m$. Then $\tau \sigma \tau^{-1} = c_1' \dots c_m'$ where $c_i' =
			(\tau(u_1^{(i)}) \dots \tau(u_{k_i}^{(i)}))$. On the other side,
			$\sigma^{-1} = c_1'' \dots c_m''$ where $c_i'' = (u_{k_i}^{(i)}
			\dots u_1^{(i)})$. Once a cycle decomposition is unique, every
			$c_i''$ must be equal to some $c_j'$. If $j \ne i$ then $c_i$ and
			$c_j$ are $\tau$-symmetric, and if $j = i$ then $c_i$ is
			$\tau$-self-symmetric. 
		\end{proof}
		
		\begin{theorem}\label{Th:Equiv}
			There exists a one-to-one correspondence between the following three
			sets:
			\begin{enumerate}
				\item\label{It:Coset} The quotient (the set of left cosets)
				$S_{2n}/B_n$ where we assume $B_n = C(\tau)$;
				
				\item\label{It:AlphaCycl} The set $B_n^{\sim}$ of permutations
				$\sigma \in C^{\sim}(\tau)$ such that their cycle decomposition
				contains no $\tau$-self-symmetric cycles.
				
				\item\label{It:Invol} The set of fixed-point-free involutions
				$\lambda \in S_{2n}$.
			\end{enumerate}
			
			The size of each set is $(2n-1)!! = 1 \times 3 \DT\times (2n-1)$.
		\end{theorem}
		
		\begin{proof}
			To prove the theorem we will construct injective maps \ref{It:Coset}
			$\to$ \ref{It:AlphaCycl} $\to$ \ref{It:Invol} $\to$ \ref{It:Coset}.
			
			\ref{It:Coset} $\to$ \ref{It:AlphaCycl}: let $\sigma \in S_{2n}$ be
			an element of the coset $\lambda \in S_{2n}/B_n$; take $Q(\lambda)
			\bydef [\sigma,\tau] \bydef \sigma \tau \sigma^{-1} \tau$. Since
			$\tau$ is an involution, one has $\tau Q(\lambda) \tau = \tau \sigma
			\tau \sigma^{-1} = Q(\lambda)^{-1}$, so $Q(\lambda) \in
			C^{\sim}(\tau)$. If $\sigma' \in \lambda$ is another element of the
			coset then $\sigma' = \sigma \rho$ where $\rho \tau = \tau \rho$ and
			therefore $[\sigma',\tau] = \sigma \rho \tau \rho^{-1} \sigma^{-1}
			\tau = \sigma \tau \rho \rho^{-1} \sigma^{-1} \tau = Q(\lambda)$, so
			the map $Q: S_{2n}/B_n \to C^{\sim}(\tau)$ is well-defined. If
			$Q(\lambda) = Q(\lambda')$ where $\lambda, \lambda' \in S_{2n}/B_n$
			are represented by $\sigma$ and $\sigma'$, respectively, then
			$\sigma \tau \sigma^{-1} \tau = \sigma' \tau (\sigma')^{-1} \tau$,
			which is equivalent to $(\sigma')^{-1} \sigma \tau = \tau
			(\sigma')^{-1} \sigma$. So $(\sigma')^{-1} \sigma \in B_n$, and
			$\lambda = \lambda'$, so $Q$ is injective.
			
			Prove that actually $Q(\lambda) \in B_n^{\sim} \subset
			C^{\sim}(\tau)$. Suppose it is not the case, that is, $Q(\lambda)$
			contains a $\tau$-self-symmetric cycle $c = (u_1 \dots u_k \tau(u_k)
			\dots \tau(u_1))$. It implies that $\tau Q(\lambda)$ has a fixed
			point $u = u_k$. On the other hand, $\tau Q(\lambda) = (\tau \sigma)
			\tau (\tau \sigma)^{-1}$ is conjugate to $\tau$ and is therefore a
			product of $n$ independent transpositions having no fixed points ---
			a contradiction.
			
			\ref{It:AlphaCycl} $\to$ \ref{It:Invol}: the condition $\tau
			\sigma \tau^{-1} = \sigma^{-1}$ is equivalent to $(\sigma\tau)^2 =
			\name{id}$. If $\sigma = c_1 c_2 \dots c_k$ then $\sigma\tau$ sends
			every element of the cycle $c_i$ to an element of its
			$\tau$-symmetric cycle $c_j$. So if $j \ne i$ for all $i$ then the
			involution $\sigma\tau$ has no fixed points. The map $\sigma \mapsto
			\sigma \tau$ is obviously injective.
			
			\ref{It:Invol} $\to$ \ref{It:Coset}: if $\psi$ is a
			fixed-point-free involution then its cycle decomposition is a
			product of $n$ independent transpositions, and therefore $\psi$
			belongs to the same conjugacy class in $S_{2n}$ as $\tau$: $\psi
			= \sigma \tau \sigma^{-1}$ for some $\sigma \in S_{2n}$. Denote by
			$R(\psi) \in S_{2n}/B_n$ the left coset containing $\sigma$. The
			equality $\sigma_1 \tau \sigma_1^{-1} = \sigma_2 \tau \sigma_2^{-1}$
			is equivalent to $(\sigma_1 \sigma_2^{-1}) \tau = \tau (\sigma_1
			\sigma_2^{-1})$, that is, $\sigma_1 \sigma_2^{-1} \in B_n$. So the
			left cosets containing $\sigma_1$ and $\sigma_2$ are the same and
			$R(\psi) \in S_{2n}/B_n$ is well-defined. If $R(\psi_1) =
			R(\psi_2)$ where $\psi_i = \sigma_i \tau \sigma_i^{-1}$, $i =
			1, 2$, then $\sigma_1 \sigma_2^{-1} \in B_n$ and therefore
			$\psi_1 = \psi_2$; thus, $R$ is an injective map.  
		\end{proof}
		
		Fix a partition $\lambda=(\lambda_1 \DT, \lambda_s)$, $\lmod
		\lambda\rmod = n$, and denote by $B_\lambda^{\sim} \subset B_n^{\sim}$
		the set of permutations whose decomposition into independent cycles
		consists of $s$ pairs of $\tau$-symmetric cycles of lengths
		$\lambda_1 \DT, \lambda_s$. Apparently, $B_n^{\sim} =
		\bigsqcup_{\lmod\lambda\rmod = n} B_\lambda^{\sim}$.
		
		\begin{proposition}\label{Pp:ConjCl}
			$B_\lambda^{\sim}$ is a $B_n$-conjugacy class in $S_{2n}$.
		\end{proposition}
		\begin{proof}
			Let $\sigma = c_1 c_1' \dots c_s c_s' \in B_n^{\sim}$ be the
			cycle decomposition where $c_i$ and $c_i'$ are $\tau$-symmetric for
			all $i$: $c_i' = \tau c_i^{-1} \tau$. Let $x \in B_n$, that is,
			$x\tau = \tau x$. Then $x \sigma x^{-1} = xc_1 x^{-1} \cdot x \tau
			c_1^{-1} \tau x^{-1} \DT\cdot xc_s x^{-1} \cdot x \tau
			c_s^{-1} \tau x^{-1}$. The permutations $\tilde c_i \bydef x c_i
			x^{-1}$ and $\tilde c_i' = x c_i' x^{-1}$ are cycles of length
			$\lambda_i$, and they are $\tau$-symmetric: $\tau \tilde c_i \tau =
			\tau x c_i x^{-1} \tau = x \tau c_i \tau x^{-1} = x (c_i')^{-1}
			x^{-1} = (\tilde c_i')^{-1}$. Thus, $x \sigma x^{-1} \in
			B_\lambda^{\sim}$.
			
			On the other side, let $\tilde\sigma = \tilde c_1 \tilde c_1' \dots
			\tilde c_s \tilde c_s' \in B_\lambda^{\sim}$. Let $\tilde c_i
			= (v_1^{(i)} \dots v_{\lambda_i}^{(i)})$, so $\tilde c_i' =
			(\tau(v_{\lambda_i}^{(i)}) \dots \tau(v_1^{(i)}))$. Define an element
			$x \in S_{2n}$ such that $x(u_s^{(i)}) = v_s^{(i)}$ and
			$x(\tau(u_s^{(i)})) = \tau(v_s^{(i)})$. Then $x \sigma x^{-1} =
			\tilde \sigma$ and $x\tau = \tau x$ (that is, $x \in B_n$).
		\end{proof}
		
		Now denote by $\mathfrak S_{m,\lambda}$ the set of decompositions into
		$m$ ribbons of the surfaces $M \in \DBS_n$ such that $\dM$ has $s$
		components containing $\lambda_1 \DT, \lambda_s$ marked points.
		
		Let $\mathcal G \in \mathfrak S_{m,\lambda}$ be a ribbon decomposition
		of $M \in \DBS_n$. Denote by $\wM \in \DBS_{2n}$ the orientation cover
		of $M$ with a ribbon decomposition given by \eqref{Eq:2Cover}. Now
		define $\Xi(\mathcal G) \bydef (\sigma_1 \DT, \sigma_m)$ where each
		$\sigma_k \bydef (i_k^{\eps_k},j_k^{\delta_k}) \in S_{2n}$ is a
		transposition; here we are using the notation of Theorem
		\ref{Th:OrientViaGraph}. Denote
		
		\begin{align*}
			\mathfrak H_{m,\lambda} &\bydef \{(\sigma_1 \DT,\sigma_m) \mid \forall s
			= 1 \DT, m \,\sigma_s = (i_s j_s), j_s \ne \tau(i_s),\\
			&\sigma_1\sigma_2 \dots \sigma_m (\tau\sigma_m\tau) \dots
			(\tau\sigma_1\tau) \in B_\lambda^{\sim}\}.
		\end{align*}
		
		\begin{theorem}\label{Th:RibbonToPermut}
			For any $\lambda$ and $m$ the map $\Xi$ is a one-to-one
			correspondence between $\mathfrak S_{m,\lambda}$ and $\mathfrak
			H_{m,\lambda}$. 
		\end{theorem}
		
		\begin{proof}
			Let $\mathcal G \in \mathfrak S_{m,\lambda}$ be a ribbon
			decomposition of $M \in \DBS_n$. By Theorem \ref{Th:OrientViaGraph}
			the diagonal of the $k$-th ribbon in the ribbon decomposition
			\eqref{Eq:2Cover} of $\wM$ joins the marked points numbered
			$i_k^{\eps_k}$ and $j_k^{\delta_k}$ if $1 \le k \le m$ and the
			points numbered $i_{k-m}^{-\eps_{k-m}}$ and
			$j_{k-m}^{-\delta_{k-m}}$ if $m+1 \le k \le 2m$. Then by Property
			\ref{It:BoundPerm} of Theorem \ref{Th:PropOrient} the boundary
			permutation of $\wM$ is
			\begin{equation*}
				\sigma = \sigma_1\sigma_2 \dots \sigma_m (\tau\sigma_m\tau) \dots
				(\tau\sigma_1\tau);
			\end{equation*}
			it has the cyclic type $(\lambda_1, \lambda_1 \DT, \lambda_s,
			\lambda_s)$ by the definition of $\mathfrak S_{m,\lambda}$. On the
			other hand, $\tau \sigma \tau = (\tau \sigma_1 \tau) \dots (\tau
			\sigma_m \tau) \sigma_m \dots \sigma_1 = \sigma^{-1}$ because
			$\sigma_k$ and $\tau \sigma_k \tau$ are involutions for all
			$k$. Thus, $\sigma \in B_n^{\sim}$, hence $\sigma \in
			B_\lambda^{\sim}$, and so $\Xi(\mathcal G) \in \mathfrak
			H_{m,\lambda}$.
			
			The map $\Xi$ is obviously one-to-one. 
		\end{proof}
		
		\begin{corollary}
			The boundary permutation of the orientation cover $\wM \in \DBS_{2n}$
			of any $M \in \DBS_n$ belongs to $B_n^{\sim}$.
		\end{corollary}
		
		\begin{corollary}
			If the twisted Hurwitz number $h_{m,\lambda}^{\sim}$ is defined by
			equation \eqref{Eq:DefHurw} then equality \eqref{Eq:HurwAlgebr}
			takes place.
		\end{corollary}
		
		\begin{example}\label{Ex:HNumbers}
			For $m=1$, any $n = \lmod\lambda\rmod$ and any transposition $\sigma
			\ne (i,i+n)$ the permutation $\mu = \sigma \tau \sigma \tau$ belongs
			to $B_{n,2^1 1^{n-2}}$. Now $\#\mathfrak H_{1,2^11^{n-2}}$ is the
			total number of all transpositions $\sigma \in S_{2n}$ except
			$(i,i+n)$, which is $\frac12 (2n)(2n-1) - n = 2n(n-1)$. So,
			$h_{1,2^11^{n-2}} = \frac{2n(n-1)}{n!} = \frac{2}{(n-2)!}$ and
			$h_{1,\lambda} = 0$ for all other $\lambda$.
			
			Let $m = 2$, $n = \lmod\lambda\rmod = 2$; here $\tau = (13)(24) \in
			S_4$. The set $B_2^{\sim} \subset B_2 = C(\tau) \subset S_4$ is a
			union of two conjugacy classes, $B_{[2]}^{\sim} =
			\{(12)(34),(14)(23)\}$ and $B_{[1,1]}^{\sim} = \{e\}$.
			
			Consider the permutation $\mu \bydef
			\sigma_1\sigma_2\tau\sigma_2\sigma_1\tau$ where $\sigma_1, \sigma_2
			\in \{(12), (14), (23), (34)\}$; totally, there are $16$ of them. It
			is easy to see that $\mu = e \in B_{[1,1]}^{\sim}$ if and only if
			$\sigma_2 = \sigma_1$ or $\sigma_2 = \tau \sigma_1 \tau$; the
			remaining $8$ pairs of transpositions $(\sigma_1, \sigma_2)$ give
			$\mu \in B_{[2]}^{\sim}$. This gives $h_{2,[1,1]}^{\sim} =
			h_{2,[2]}^{\sim} = \frac{8}{2!}=4$.
			
			For $n=3$, $m=2$ the calculations (in $S_6$) are similar but more
			cumbersome, giving eventually $h_{2,[1,1,1]}^{\sim} =
			h_{2,[2,1]}^{\sim} = 4$ and $h_{2,[3]}^{\sim} = 16$.
		\end{example}
		
		\subsection{Twisted cut-and-join operator}
		
		Now denote
		\begin{equation}\label{Eq:ClassSum}
			\CSum{\lambda} \bydef \sum_{\sigma\in B_\lambda^{\sim}}\sigma \in
			\Complex[B_n^{\sim}].
		\end{equation}
		(a conjugacy class sum). Also, call the set
		\begin{equation*}
			\mathcal{Z}(B_n^{\sim}) \bydef \{y \in \Complex[B_n^{\sim}] \mid x y x^{-1} =
			y\, \forall x \in B_n\}
		\end{equation*}
		a {\em twisted center} of $B_n$. It is clear that $\CSum{\lambda}$
		belong to $\mathcal{Z}(B_n^{\sim})$ and form a basis in it.
		
		Let $\Complex[p]$ be a space of polynomials of the countable set of
		variables $p = (p_1, p_2, \dots)$. Assume $\deg p_k = k$ for all $k$
		and denote by $\Complex[p]_n$ the space of homogeneous polynomials of
		degree $n$. A linear map $\Psi: \mathcal{Z}(B_n^{\sim}) \to
		\Complex[p]_n$ defined by
		\begin{equation}\label{Eq:Center}
			\Psi(\CSum{\lambda}) = p_\lambda \bydef p_{\lambda_1} \dots
			p_{\lambda_s}
		\end{equation}
		is obviously an isomorphism of vector spaces.
		
		Define an operator $\CJfrak: \mathcal{Z}(B_n^{\sim}) \to
		\mathcal{Z}(B_n^{\sim})$ by
		\begin{equation*}
			\CJfrak(\sigma) = \sum_{\substack{1 \le i < j \le 2n\\j \ne
					\tau(i)}} (ij) \sigma (\tau(i)\tau(j))
		\end{equation*}
		
		\begin{definition}\label{Df:CJ}
			The {\em twisted cut-and-join operator} is a linear map $\CJ:
			\Complex[p]_n \to \Complex[p]_n$ making the following diagram
			commutative:
			\begin{equation}\label{Eq:CJViaGAlg}
				\xymatrix{
					\mathcal{Z}[B_n^{\sim}] \ar[r]^{\CJfrak} \ar[d]^\Psi &
					\mathcal{Z}[B_n^{\sim}] \ar[d]^\Psi \\
					\Complex[p]_n \ar[r]_{\CJ} & \Complex[p]_n}
			\end{equation}
		\end{definition}
		
		Let $\lambda, \mu$ be partitions such that $\lmod\lambda\rmod =
		\lmod\mu\rmod = n$. Take an element $\sigma \in B_\lambda^{\sim}$ and
		consider a set
		\begin{equation*}
			S(\sigma;\mu) \bydef \{(i,j) \mid \le i, j \le 2n, j \ne i,
			\tau(i), (ij)\sigma_* (\tau(i)\tau(j)) \in B_{\mu}^{\sim}\}.
		\end{equation*}
		Proposition \ref{Pp:ConjCl} implies that for every $x \in B_n$ and
		$\sigma \in B_\lambda^{\sim}$ the map $(i,j) \mapsto (x(i),x(j))$ is a
		bijection between $S(x \sigma x^{-1},\mu)$ and
		$S(\sigma,\mu)$. So, the size of the set $S(\sigma,\mu)$ for
		$\sigma \in B_\lambda^{\sim}$ depends on $\lambda$ and $\mu$
		only.
		
		We will be using ``physical'' notation for matrix elements of a linear
		operator: $\CJfrak(\CSum{\lambda}) = \sum_{\mu} \langle \lambda
		\mid \CJfrak \mid \mu\rangle \CSum{\mu}$
		
		\begin{theorem}\label{Th:TViaMult}
			$\langle \lambda \mid \CJfrak \mid \mu\rangle = \frac12
			\#S(\sigma,\mu)$ for any $\sigma \in B_\lambda^{\sim}$. 
		\end{theorem}
		
		\begin{proof}
			By definition,
			\begin{equation}\label{Eq:ActionCJ}
				\CJfrak(\CSum{\lambda}) = \sum_{\sigma \in B_{\lambda}^{\sim}}
				\CJfrak(\sigma) = \sum_{\sigma \in
					B_\lambda^{\sim}}\sum_{\substack{1\leq i<j\leq 2n\\ j \ne
						\tau(i)}} (ij) \sigma(\tau(i)\tau(j)).
			\end{equation}
			As it was noted above, \eqref{Eq:ActionCJ} is a sum of identical
			summands, so
			\begin{equation*}
				\CJfrak(\CSum{\lambda}) = \#B_\lambda^{\sim} \sum_{\mu}
				\sum_{\substack{1\leq i<j\leq 2n\\ j \ne \tau(i)\\ (ij)
						\sigma(\tau(i)\tau(j)) \in B_{\mu}^{\sim}}} (ij) \sigma(\tau(i)\tau(j)).
			\end{equation*}
			for any fixed $\sigma \in B_\lambda^{\sim}$, and therefore
			\begin{align*}
				\CJfrak(\CSum{\lambda}) &= \sum_{\mu} \sum_{\substack{1\leq
						i<j\leq 2n\\ j \ne \tau(i)\\ (ij) \sigma(\tau(i)\tau(j))
						\in B_{\mu}^{\sim}}} \sum_{\tau \in B_{\mu}^{\sim}} \tau\\    
				&= \frac12 \sum_{\mu} \#\{(i,j) \mid j \ne i, \tau(i), 
				(ij)\sigma(\tau(i)\tau(j)) \in B_{\mu}^{\sim}\}
				\CSum{\mu}.
			\end{align*}
		\end{proof}
		
		Consider the generating function $\mathcal{H}^{\sim}(\beta,p)$ of the
		twisted Hurwitz numbers defined as follows:
		\begin{equation}\label{Eq:GFHurw}
			\mathcal{H}^{\sim}(\beta,p) = \sum_{m\geq 0} \sum_\lambda
			\frac{h_{m,\lambda}^{\sim}}{m!} p_{\lambda_1} p_{\lambda_2} \dots
			p_{\lambda_s} \beta^m.
		\end{equation}
		
		\begin{theorem}\label{Th:HurwCJ}
			$\mathcal{H}^{\sim}$ satisfies the cut-and-join equation
			$\pder{\mathcal{H}^{\sim}}{\beta} = \CJ(\mathcal{H}^{\sim})$.
		\end{theorem}
		
		\begin{proof}
			Fix a positive integer $n$ and denote by $\mathcal{H}_n^{\sim}$ a degree
			$n$ homogeneous component of $\mathcal{H}^{\sim}$. The twisted
			cut-and-join operator preserves the degree, so $\mathcal{H}^{\sim}$
			satisfies the cut-and-join equation if and only if $\mathcal{H}_n^{\sim}$
			does (for each $n$).
			
			Let
			\begin{equation*}
				\mathcal{G}_n \bydef \sum_{m \ge 0} \sum_{\lambda:
					\lmod\lambda\rmod = n}
				\frac{n!h_{m,\lambda}^{\sim}}{m!} \CSum{\lambda} \beta^m \in
				\Complex[S_{2n}]
			\end{equation*}
			where $\CSum{\lambda}$ is defined by \eqref{Eq:ClassSum}. An
			elementary combinatorial reasoning gives
			\begin{equation*}
				\mathcal{G}_n = \sum_{m\geq 0}
				\frac{\beta^m}{m!}(\CJfrak)^m(e_{2n})
			\end{equation*}
			where $e_{2n} \in S_{2n}$ is the unit element. Clearly
			$\CJfrak(\mathcal{G}_n) = \sum_{m\geq 0}
			\frac{\beta^m}{m!}(\CJfrak)^{m+1}(e_{2n}) = \sum_{m \ge 1}
			\frac{\beta^{m-1}}{(m-1)!} (\CJfrak)^{m}(e_{2n}) =
			\pder{\mathcal{G}_n}{\beta}$. Applying $\Psi$ one obtains $\Psi \CJfrak(\mathcal{G}_n) =
			\Psi(\pder{\mathcal{G}_n}{\beta}) = \pder{}{\beta}
			\Psi(\mathcal{G}_n)$.  By \eqref{Eq:Center}, $\Psi(\mathcal{G}_n) =
			\mathcal{H}_n^{\sim}$, hence $\pder{}{\beta} \Psi(\mathcal{G}_n) =
			\pder{\mathcal{H}^{\sim}_n}{\beta}$. By the definition of the twisted
			cut-and-join operator, $\Psi \CJfrak(\mathcal{G}_n) =
			\CJ(\Psi(\mathcal{G}_n)) = \CJ(\mathcal{H}^{\sim}_n)$, and the equality
			$\pder{\mathcal{H}^{\sim}_n}{\beta} = \CJ(\mathcal{H}^{\sim}_n)$ follows. 
		\end{proof}
		
		\begin{corollary}\label{Cr:HurwExp}
			$\mathcal{H}^{\sim}(\beta,p) = \exp(\beta \CJ) \exp(p_1)$.
		\end{corollary}
		
		\begin{proof}
			It follows from \eqref{Eq:DefHurw} that $h_{0,\lambda} =
			\frac{1}{n!}$ if $\lambda = 1^n$ and $h_{0,\lambda} = 0$
			otherwise. Thus, $\mathcal{H}^{\sim}(0,p) = \exp(p_1)$, and the
			formula follows from Theorem \ref{Th:HurwCJ}. 
		\end{proof}
		
		\subsection{Explicit formulas}
		
		In this section we prove explicit formulas for the cut-and-join
		operator (Theorem \ref{Th:CJtwisted}) and for twisted Hurwitz numbers
		(Theorem \ref{Th:HurwExpl}).
		
		\begin{theorem}\label{Th:CJtwisted}
			The twisted cut-an-join operator is given by 
			\begin{align}
				\CJ &= \sum_{i,j \ge 1} (i+j) p_i p_j \pder{}{p_{i+j}} + 2ij
				p_{i+j} \pdertwo{}{p_i}{p_j} + \sum_{k \ge 1} k(k-1) p_k
				\pder{}{p_k}\nonumber\\
				&= \sum_{i,j \ge 1} (i+j) (p_i p_j + p_{i+j}) \pder{}{p_{i+j}}
				+ 2ij p_{i+j} \pdertwo{}{p_i}{p_j}\label{Eq:CJtwisted}
			\end{align}
		\end{theorem}
		
		\noindent (The two formulas are equivalent because there are $k-1$
		pairs $(i,j)$ such that $i, j \ge 1$ and $i+j = k$.)
		
		To prove Theorem \ref{Th:CJtwisted} we calculate explicitly the matrix
		elements $\langle \lambda \mid \CJfrak \mid \mu\rangle$ for all
		possible $\lambda, \mu$.
		
		Let $\sigma \in S_n$ and $1 \le i < j \le n$. The cycle structure of
		the product $\sigma' = (ij)\sigma$ depends on the cycle structure of
		$\sigma$ and on $i$ and $j$ as follows: if $i$ and $j$ belong to the
		same cycle $(x_1 \DT, x_\ell)$ of $\sigma$ (where $i = x_1$, $j =
		x_m$), then $\sigma'$ contains two cycles $(x_1 \DT, x_{m-1})$ and
		$(x_m \DT, x_\ell)$ instead (``a cut''). If $i$ and $j$ are in
		different cycles $(x_1 \DT, x_m)$ and $(y_1 \DT, y_k)$ (where $i =
		x_1$ and $j = y_1$) then $\sigma'$ contains the cycle $(x_1 \DT, x_m,
		y_1 \DT, y_k)$ instead (``a join'').
		
		Let now $\sigma \in B_\lambda^{\sim} \subset B_n^{\sim}$ where
		$\lambda = 1^{a_1}2^{a_2} \dots n^{a_n}$ (in other words, the element
		$\sigma \in S_{2n}$ contains $a_s$ pairs of $\tau$-symmetric cycles of
		length $s$ for $s = 1 \DT, n$). Let $1 \le i < j \le 2n$, $j \ne
		\tau(i)$ and $\sigma' \bydef (ij)\sigma(\tau(i)\tau(j)) \in
		B_{\mu}^{\sim}$. The cyclic structure of $\sigma'$ depends on the
		positions of $i$, $j$, $\tau(i)$, $\tau(j)$ and on the cycles of
		$\sigma$; there are three possible cases shown in
		Fig.~\ref{Fg:3cases}.
		
		\begin{figure}
			\includegraphics[scale=0.6]{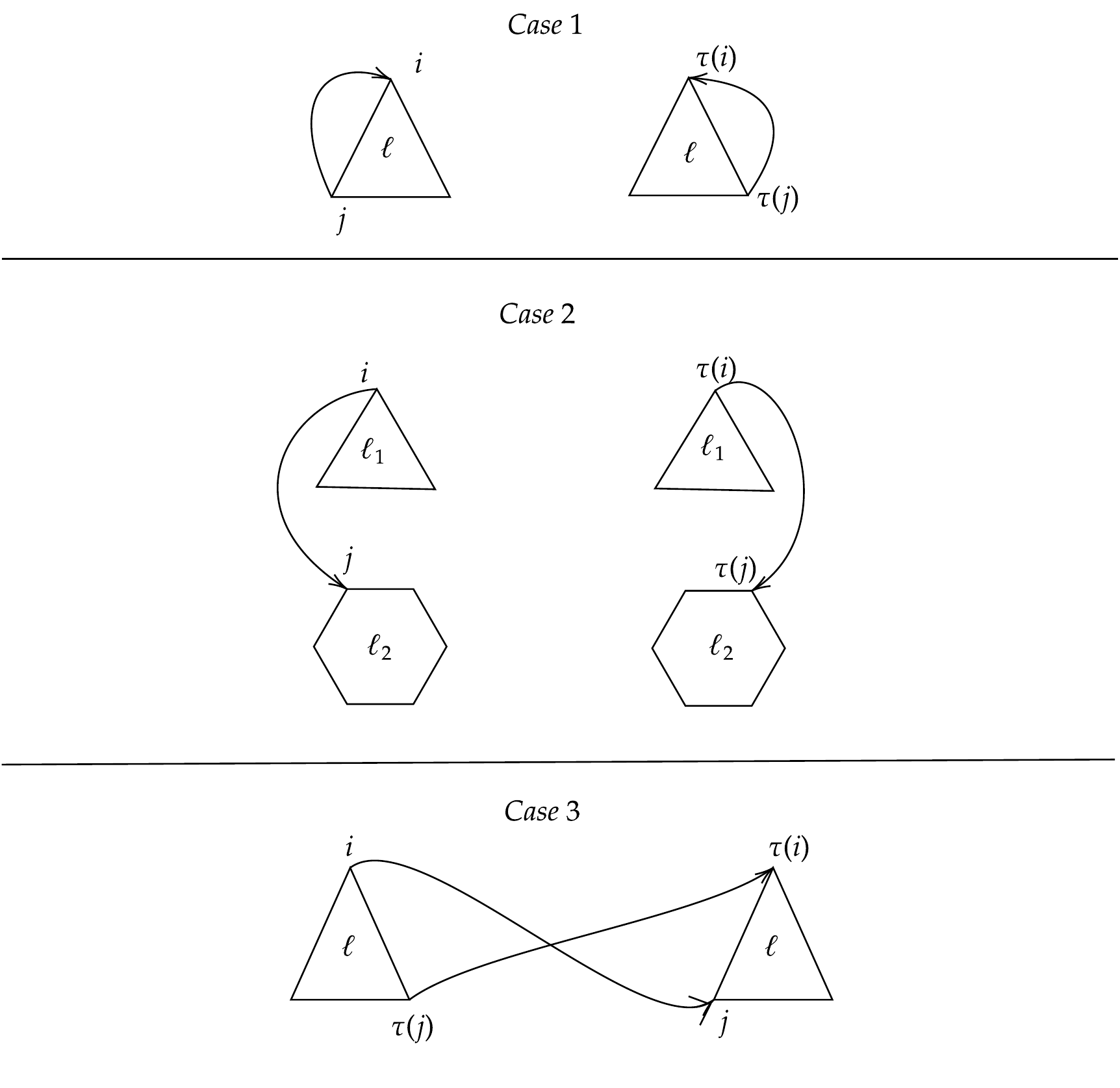} 
			\caption{Terms of $\CJfrak$}\label{Fg:3cases}
		\end{figure}
		
		{\def \labelenumi {Case \theenumi.}
			\begin{enumerate}
				\item Here $\mu$ is obtained from $\lambda$ by a cut:
				\begin{alignat*}{2}
					\mu &= 1^{a_1} \dots m^{a_m+1} \dots
					k^{a_k+1} \dots \ell^{a_\ell-1} \dots n^{a_n},
					&&\quad m+k=\ell, m < k,\\
					&\text{or}\\
					\mu &= 1^{a_1} \dots m^{a_m+2} \dots \ell^{a_\ell-1}
					\dots n^{a_n}, &&\quad m = \ell/2.
				\end{alignat*}
				For a fixed $\sigma \in B_\lambda^{\sim}$ look for $i, j$ such
				that $\sigma' \bydef (ij)\sigma(\tau(i)\tau(j)) \in
				B_{\mu}^{\sim}$. The element $\sigma$ contains $2a_\ell$
				cycles of length $\ell$, so there are $2\ell a_\ell$ possible
				positions for $i$. In $\sigma'$ the elements $i$ and $j$ are in
				different cycles; if $m < k$ then $m$ may be the
				length of either. So if $m < k$ then $j$ should be in
				the same cycle in $\sigma$ as $i$, and the distance between them
				is either $m$ or $k$. So there are two possible
				positions for $j$ once $i$ is chosen, and $\langle \mu \mid
				\CJfrak \mid \lambda \rangle = \frac12 \#S(\sigma,\mu) =
				2\ell a_\ell$. If $m = k = \ell/2$ then the position for
				$j$ is unique and $\langle \mu \mid \CJfrak \mid \lambda
				\rangle = \ell a_\ell$.
				
				\item Here $\mu$ is obtained from $\lambda$ by a join:
				\begin{alignat*}{2}
					\mu &= 1^{a_1} \dots m^{a_m-1} \dots
					k^{a_k-1} \dots \ell^{a_\ell+1} \dots n^{a_n},
					&&\quad m+k=\ell, m < k,\\
					&\text{or}\\
					\mu &= 1^{a_1} \dots m^{a_m-2} \dots
					\ell^{a_\ell+1} \dots n^{a_n}, &&m = \ell/2
				\end{alignat*}
				If $m<k$ then $i$ may belong to the cycle of either length. If $i$
				belongs to the cycle of length $m$ then there are $2ma_m$ possible
				positions for it (cf.\ Case 1) and $2ka_k$ positions for $j$;
				vice versa if $i$ belongs to the cycle of length $k$. The matrix
				element is then $\langle \mu \mid \CJfrak \mid \lambda
				\rangle = 4mk a_m a_k$. If $m=k = \ell/2$ then $i$ and $j$ belong
				to cycles of the same length $m$; the cycle containing $j$
				contains neither $i$ nor $\tau(i)$. Hence there are $4a_m (a_m-1)$
				possibilities for choosing a pair of cycles to contain $i$ and
				$j$, and $m^2$ possible positions for $i$ and $j$ in
				them. Therefore $\langle \mu \mid \CJfrak \mid \lambda\rangle
				= 2m^2a_m(a_m-1)$.
				
				\item Here $\mu = \lambda$. Like in the previous cases we have
				$2\ell a_\ell$ possible positions for $i$ and $\ell-1$ positions
				for $j \ne \tau(i)$ (in the cycle $\tau$-symmetric to the one
				containing $i$) once $i$ is fixed. Thus, $\langle \mu \mid
				\CJfrak \mid \lambda\rangle = \sum_\ell 2\ell(\ell-1) a_\ell$.  
		\end{enumerate}}
		
		\begin{proof}[Proof of Theorem \ref{Th:CJtwisted}]
			It follows from Theorem \ref{Th:TViaMult} and Definition \ref{Df:CJ}
			that $\CJ p_\lambda = \sum_\mu \langle\lambda \mid \CJfrak
			\mid \mu\rangle p_\mu$.
			
			For a given $\lambda$ there are three types of $\mu$ such that
			$\langle\lambda \mid \CJfrak \mid \mu\rangle \neq 0$ listed
			above. Hence $\CJ$ is a sum of three terms.
			
			Suppose $\mu$ is like in Case 1 with $m < k$. The monomial
			$p_\lambda$ contains $p_m^{a_m} p_k^{a_k} p_\ell^{a_\ell}$ and the
			monomial $p_\mu$ contains $p_m^{a_m+1} p_k^{a_k+1}
			p_\ell^{a_\ell-1}$; the other factors are the same. So the term in
			\eqref{Eq:CJtwisted} acting on $p_\lambda$ and giving $p_\mu$
			is $2\ell p_m p_k \pder{}{p_\ell} p_\lambda = 2\ell a_\ell
			p_\mu = \langle \mu \mid \CJfrak \mid \lambda \rangle
			p_\mu$ (actually there are two equal terms in the sum: $i =
			m, j = k$ or vice versa, hence the factor $2$).
			
			If $\mu$ is like in Case 1 with $m = \ell/2$ then $p_\lambda$
			contains $p_{\ell/2}^{a_{\ell/2}} p_\ell^{a_\ell}$ and $\mu$
			contains $p_{\ell/2}^{a_{\ell/2}+2} p_\ell^{a_\ell-1}$. So the only
			term in \eqref{Eq:CJtwisted} acting on $p_\lambda$ and giving
			$p_\mu$ is $\ell p_{\ell/2}^2 \pder{}{p_\ell} p_\lambda =
			\ell a_\ell p_\mu = \langle \mu \mid \CJfrak \mid
			\lambda \rangle p_\mu$.
			
			The calculations for the two remaining cases are similar.
		\end{proof}
		
		By Theorem \ref{Th:CJtwisted}, $\CJ = \CJcl + \Tw$ where
		\begin{equation*}
			\CJcl = \sum_{i,j \ge 1} (i+j) p_i p_j \pder{}{p_{i+j}} +
			ij p_{i+j} \pdertwo{}{p_i}{p_j}
		\end{equation*}
		is the classical cut-and-join, and
		\begin{equation*}
			\Tw = \sum_{i,j \ge 1} p_{i+j} ((i+j) \pder{}{p_{i+j}} + ij
			\pdertwo{}{p_i}{p_j}).
		\end{equation*}
		
		A one-parametric family
		\begin{equation*}
			\Delta_\alpha \bydef \CJcl + (\alpha-1)\Tw = \sum_{i,j \ge 1} (i+j)
			(p_i p_j + (\alpha-1) p_{i+j}) \pder{}{p_{i+j}} + \alpha ij p_{i+j}
			\pdertwo{}{p_i}{p_j}
		\end{equation*}
		is called \cite{Macdonald} the Laplace-Beltrami operator; in
		particular, $\Delta_1 = \CJcl$ is the classical cut-and-join and
		$\Delta_2 = \CJ$, the twisted cut-and-join. By the classical results
		of \cite[p.~376 and after]{Macdonald}, the eigenvalues (and
		eigenvectors) of $\Delta_\alpha$ are indexed by partitions. The
		eigenvalue corresponding to $\lambda = (\lambda_1 \DT\ge \lambda_s)$
		is equal to
		\begin{equation*}
			e(\lambda,\alpha) = \sum_{i=1}^s \lambda_i(\alpha \lambda_i +
			2-2i-\alpha).
		\end{equation*}
		The corresponding eigenvector is a polynomial
		$J_\lambda^{(\alpha)}(p)$ of degree $\lmod\lambda\rmod \bydef
		\lambda_1 \DT+ \lambda_s$ called Jack polynomial; it is normalized so
		that the coefficient at $p_1^n$ in it is $1$. Polynomials $Z_\lambda
		\bydef J_\lambda^{(2)}$ are called zonal.
		
		\begin{theorem}[\cite{Macdonald}]
			\begin{equation*}
				\sum_\lambda
				\frac{J_\lambda^{(\alpha)}(p)J_\lambda^{(\alpha)}(q)}{H_\lambda(\alpha)
					H'_\lambda(\alpha)} = \exp(\sum_{k \ge 1} \frac{p_k
					q_k}{k\alpha}).
			\end{equation*}
			where $H_\lambda(\alpha) \bydef \prod_{(i,j) \in Y(\lambda)} (\alpha
			a(i,j) + \ell(i,j) + 1)$ and $H'_\lambda(\alpha) \bydef \prod_{(i,j)
				\in Y(\lambda)} (\alpha a(i,j) + \ell(i,j) + \alpha)$. Here
			$Y(\lambda)$ is the Young diagram of the partition $\lambda$, and
			$a(i,j)$ and $\ell(i,j)$ are the arm and the leg, respectively, of
			the cell $(i,j) \in Y(\lambda)$.
		\end{theorem}
		
		Substituting $q_1 = \alpha, q_2 = q_3 \DT= 0$ and taking into account
		the normalization of the Jack polynomials one obtains
		\begin{corollary}\label{Cr:Cauchy}
			$\sum_\lambda \frac{\alpha^{\lmod\lambda\rmod}
				J_\lambda^{(\alpha)}(p)}{H_\lambda(\alpha) H'_\lambda(\alpha)} =
			\exp(p_1)$.
		\end{corollary}
		
		Taking now $\alpha=2$ and substituting the formula of Corollary
		\ref{Cr:Cauchy} into Corollary \ref{Cr:HurwExp}, one obtains
		\begin{theorem}\label{Th:HurwExpl}
			\begin{equation*}
				\mathcal{H}^{\sim}(\beta,p) = \sum_{\lambda} \exp\bigl(2\beta \sum_i
				\lambda_i (\lambda_i-i)\bigr) \frac{2^{\lmod\lambda\rmod}
					Z_\lambda(p)}{H_\lambda(2) H'_\lambda(2)}.
			\end{equation*}
		\end{theorem}
		This is a twisted analog of the formula expressing the usual Hurwitz
		numbers via Schur polynomials, see \cite{LandoKazarian}.
		
		\begin{example}
			The zonal polynomials $Z_\lambda$ for small $\lambda$ are:
			\begin{trivlist}
				\item $Z_{[1]}=p_1,$ with $H_{[1]}(2)H_{[1]}'(2)=2$
				\item $Z_{[1,1]}=p_1^2-p_2$ with $H_{[1,1]}(2)H_{[1,1]}'(2)=12$
				\item $Z_{[2]}=p_1^2+2p_2$ with $H_{[2]}(2)H_{[2]}'(2)=24$
				\item $Z_{[1,1,1]}=p_1^3-3p_2p_1+2p_3$ with
				$H_{[1,1,1]}(2)H_{[1,1,1]}'(2)=144$
				\item $Z_{[2,1]}=p_1^3+p_2p_1-2p_3$ with with
				$H_{[2,1]}(2)H_{[2,1]}'(2)=80$
				\item $Z_{[3]}=p_1^3+6p_2p_1+8p_3$ with $H_{[3]}(2)H_{[3]}'(2)=720$
			\end{trivlist}
			
			This gives us the first few terms in the expansion of
			$\mathcal{H}^{\sim}(\beta,p)$:
			\begin{align*}
				\mathcal{H}^{\sim}(\beta,p) &= p_1 +
				\frac{p_1^2}{6}(2e^{-2\beta}+e^{4\beta}) +
				\frac{p_2}{3}(-e^{-2\beta}+e^{4\beta}) +
				\frac{p_1^3}{90}(9e^{2\beta}+e^{12\beta}+5e^{-6\beta})\\
				&+ \frac{p_2p_1}{30}(2e^{12\beta}+3e^{2\beta}-5e^{-6\beta})
				+\frac{p_3}{45}(4e^{12\beta}-9e^{2\beta}+5e^{-6\beta}) + \dots\\
				&= (p_1 + \frac{p_1^2}{2} + \frac{p_1^3}{6} + \dots) + \beta (2p_1 + 2p_1
				p_2 + \dots)\\
				&\hphantom{= (p_1 + \frac{p_1^2}{2} + \frac{p_1^3}{6} + \dots)} +
				\beta^2 (2p_1^2 + 2p_2 + 2p_1^3 + 2p_1 p_2 + 8p_3 + \dots) + \dots ;
			\end{align*}
			they agree with \eqref{Eq:GFHurw} and Example \ref{Ex:HNumbers}.
		\end{example}
		
		\section{Algebro-geometric model: twisted branched coverings}
		
		A classical notion of the branched covering was extended to the
		non-orientable case by G.\,Chapuy and M.\,Dołęga in \cite{CD}. Let $N$
		denote a closed surface (compact $2$-manifold without boundary, not
		necessarily orientable), and $p: \wN \to N$, its orientation cover. As
		above, denote by $\T: \wN \to \wN$ an orientation-reversing involution
		without fixed points such that $p \circ \T = p$. Also denote by
		$\mathcal J: \Complex P^1 \to \Complex P^1$ the complex conjugation,
		and let $\Upp \bydef \Complex P^1/(z \sim {\mathcal J}(z)) = \mathbb H
		\cup \{\infty\}$ where $\mathbb H \subset \Complex$ is the upper
		half-plane; $\Upp$ is homeomorphic to a disk. Denote by $\pi: \Complex
		P^1 \to \Upp$ the quotient map.
		
		\begin{definition}[\cite{CD}]\label{Df:Branched}
			A continuous map $f: N \to \Upp$ is called a twisted branched
			covering if there exists a branched covering $\wf: \wN \to \Complex
			P^1$ such that
			\begin{enumerate}
				\item\label{It:MapReal} $\pi \circ \wf = f \circ p$, and
				\item all the critical values of $\wf$ are real.
			\end{enumerate}
		\end{definition}
		
		Property (\ref{It:MapReal}) is equivalent to saying that $\wf$ is a
		real map with respect to $\T$, that is, $\wf \circ \T = {\mathcal J}
		\circ \wf$. The involution $\T$ has no fixed points, so the critical
		points of $\wf$ come in pairs $(a, \T(a))$, the ramification profile
		of every critical value $c \in \Real P^1 \subset \Complex P^1$ of
		$\wf$ has every part repeated twice: $(\lambda_1, \lambda_1 \DT,
		\lambda_s, \lambda_s)$, and $\deg \wf = 2n$ is even. In this
		case we say that the ramification profile of the critical value
		$\pi(c) \in \partial\Upp$ of the map $f: N \to \Upp$ is $\lambda =
		(\lambda_1 \DT, \lambda_s)$.
		
		The twisted branched covering $f$ is called simple if all its critical
		values, except possibly $\infty \in \Upp$, have the ramification
		profile $2^1 1^{n-2}$. (Equivalently, each critical value of $\wf$ has
		$2$ simple critical points and $2n-4$ regular points as preimages.)
		Let $u \in \partial\Upp$ be a regular (not critical) value of $f$;
		then the preimage $f^{-1}(u) \subset N$ consists of $n$ points. Fix a
		bijection $\nu: f^{-1}(u) \to \{1 \DT, n\}$ (a labeling); then the
		triple $(f,u,\nu)$ is called a labeled simple twisted branched covering.
		
		Labeled simple twisted branched coverings are split into equivalence
		classes via right-left equivalence: $(f_1,u_1,\nu_1) \sim
		(f_2,u_2,\nu_2)$ if there exist orientation-preserving diffeomorphisms
		$D_1: \wN \to \wN$ and $D_2: \Complex P^1 \to \Complex P^1$ such that
		\begin{itemize}
			\item ($f_1$ transforms to $f_2$) $\wf_1 \circ D_1 = D_2 \circ \wf_2$,
			\item ($D_1$ and $D_2$ are equivariant) $\T \circ D_1 = D_1 \circ \T$
			and $D_2 \circ {\mathcal J} = {\mathcal J} \circ D_2$,
			\item ($D_1, D_2$ preserve labeling) $D_2(\pi^{-1}(u_1)) =
			\pi^{-1}(u_2)$ and $\nu_2 \circ D_1 = \nu_1$.
		\end{itemize}
		
		For an integer $m \ge 0$ and a partition $\lambda$ denote by
		${\mathfrak D}_{m,\lambda}$ the set of equivalence classes of labeled
		simple twisted branched coverings having $m$ simple critical values
		and such that the ramification profile of $\infty$ is $\lambda$.
		
		\begin{theorem}
			$\#\mathfrak D_{m,\lambda} = \#\mathfrak S_{m,\lambda} = \#\mathfrak
			H_{m,\lambda} = n!h^{\sim}_{m,\lambda}$.
		\end{theorem}
		
		\begin{proof}
			The generating function $\mathcal{D}(\beta,p) \bydef = \sum_{m\geq
				0} \sum_\lambda \frac{\#\mathfrak D_{m,\lambda}}{n!m!}
			p_{\lambda_1} p_{\lambda_2} \dots p_{\lambda_s} \beta^m$ is shown in
			\cite[Theorem 6.5 for $b=1$]{CD} to satisfy the twisted cut-and-join
			equation $\pder{\mathcal{D}}{\beta} = \CJ(\mathcal D)$ where $\CJ$
			is given by equation \eqref{Eq:CJtwisted}.
			
			Let $m=0$, so the branched covering $\wf \in \mathfrak
			D_{0,\lambda}$ is unramified except possibly over $\infty$. Denote
			by $N_0 \subset \wN$ any connected component of $\wN$, by $n_0$, the
			degree of $\left.\wf\right|_{N_0}$, and $k \bydef
			\#\left.\wf\right|_{N_0}^{-1}(\infty)$. Then the Euler
			characteristic $\chi(\Complex P^1 \setminus \{\infty\}) = 1$ and
			therefore $\chi(N_0 \setminus \wf^{-1}(\infty)) = \chi(N_0) - k =
			n_0$. The set $N_0$ is a smooth compact $2$-manifold, so $2 \ge
			\chi(N_0) = n_0+k$, implying $n_0 = k = 1$. It means that $\wf$ is
			unramified over $\infty$, too, so $\lambda = 1^n$ and $\wf$ is a
			collection of $n$ orientation-preserving diffeomorphisms of
			spheres. Obviously, $\wf$ is unique up to the right-left equivalence
			described above. Thus, $\#\mathfrak D_{0,1^n} = 1$ and $\#\mathfrak
			D_{0,\lambda} = 0$ for other $\lambda$. So, $\mathcal D(0,p) =
			\exp(p_1)$, and Corollary \ref{Cr:HurwExp} implies that $\mathcal
			D(\beta,p) \equiv \mathcal H^{\sim}(\beta,p)$ proving the theorem.
		\end{proof}
		
		\begin{Remark}
			Note that unlike Theorem \ref{Th:RibbonToPermut} we do not know any
			``natural'' one-to-one map between the sets $\mathfrak D_{m,\lambda}$
			and $\mathfrak S_{m,\lambda}$ (or $\mathfrak
			H_{m,\lambda}^{\sim}$). Finding one is a challenging topic of future
			research.
			
			In \cite{CD}, a one-parametric generalization of Hurwitz numbers is
			defined by counting twisted branched coverings with
			parameter-dependent weights. The parameter value $b=0$ gives
			classical Hurwitz numbers, and $b=1$, twisted Hurwitz numbers. A
			natural one-to-one correspondence between $\mathfrak D_{m,\lambda}$
			and $\mathfrak S_{m,\lambda}$ would allow to transfer these weights 
			to ribbon decompositions and to define parametric Hurwitz numbers
			using them.
			
			Note also that in \cite{CD}, more general two-part Hurwitz numbers
			were studied; currently we do not know other models for them.
		\end{Remark}

	\end{document}